\documentclass{amsart}
\usepackage{chngcntr}
\usepackage{apptools}
\usepackage{color}
\usepackage{amsthm,amssymb,verbatim}
\usepackage{graphicx}

\usepackage{enumerate}
\usepackage[latin1]{inputenc} 
\usepackage{graphicx}

\newcommand{\pdf}[2]{\frac{\partial #1}{\partial #2}}

\newcommand{\nn}{\eta}

\newcommand{\Rr}{\mathcal R}

 \newcommand{\Gg}{\mathcal{G}}
 \newcommand{\Nn}{\mathcal{N}}

  \newcommand{\Ss}{\mathbf{S}}
 
 \newcommand{\RR}{\mathbf{R}}  
 \newcommand{\ZZ}{\mathbf{Z}}  
 \newcommand{\BB}{\mathbf{B}}  
    
  \newcommand{\Div}{\operatorname{Div}}
  
    \newcommand{\dist}{\operatorname{dist}}
 
 \newcommand{\area}{\operatorname{area}}
 \newcommand{\eps}{\epsilon}
 \newcommand{\Tan}{\operatorname{Tan}}

\newcommand{\vv}{\mathbf v}
\newcommand{\ee}{\mathbf e}

\newcommand{\D}{{\rm D}}

\newcommand{\graph}{\textrm{graph}}
\newcommand{\spt}{\operatorname{spt}}

\usepackage{amsthm}

\input epsf
\def\begfig {
\begin{figure}
\small }
\def\endfig {
\normalsize
\end{figure}
}

    \newtheorem{theorem}    {Theorem}   
    \newtheorem{lemma}      [theorem]       {Lemma}
    \newtheorem{corollary}  [theorem]     {Corollary}
    \newtheorem{proposition}       [theorem]       {Proposition}
    
    \newtheorem{claim}{Claim}
    \newtheorem*{theorem*}{Theorem}
    \theoremstyle{definition}
    \newtheorem{definition}  [theorem] {Definition}
     \newtheorem{conjecture}  [theorem] {Conjecture}
    \theoremstyle{definition}
    \newtheorem{remark}   [theorem]       {Remark}

\usepackage{hyperref}
\usepackage[alphabetic, msc-links, backrefs]{amsrefs}
\begin{document}



\title[Tridents]{Nguyen's Tridents and 
the Classification of Semigraphical Translators for Mean Curvature Flow}

\begin{abstract}
We construct a one-parameter family of singly periodic translating solutions to mean
curvature flow that converge as the period tends to $0$ to the union of a grim reaper surface
and a plane that bisects it lengthwise.   The surfaces are semigraphical: they are properly embedded,
and, after removing a discrete 
collection of vertical lines, they are graphs.  
We also provide a nearly complete classification of semigraphical translators.
\end{abstract}

\author[D. Hoffman]{\textsc{D. Hoffman}}

\address{David Hoffman\newline
Stanford University Department of Mathematics \newline
  Stanford, CA 94305, USA\newline
{\sl E-mail address:} {\bf dhoffman@stanford.edu}
}

\author[F. Martin]{\textsc{F. Martín}}

\address{Francisco Martín\newline
Departmento de Geometría y Topología. IMAG \newline
Universidad de Granada\newline
18071 Granada, Spain\newline
{\sl E-mail address:} {\bf fmartin@ugr.es}
}
\author[B. White]{\textsc{B. White}}

\address{Brian White\newline
Stanford University Department of Mathematics \newline
  Stanford, CA 94305, USA\newline
{\sl E-mail address:} {\bf bcwhite@stanford.edu}
}

\date{September 23, 2019.  Revised December 29, 2021}
\thanks{The second author was partially supported by the MCIN/AEI/10.13039/501100011033/  grant no. PID2020-116126-I00 and  by the Regional Government of Andalusia and ERDEF grant PY20-01391. The third author was partially supported by NSF grant DMS-1711293}

\maketitle

\section{Introduction}

A {\bf translator} in $\RR^3$ is a smooth surface $M$ such that $t\mapsto M-t\ee_3$
is a mean curvature flow, or, equivalently, such that the mean curvature vector at each point of $M$ is 
given by $(-\ee_3)^\perp$. Ilmanen~\cite{ilmanen} observed that $M\subset\RR^3$ is a translator
if and only if $M$ is a minimal surface 
 with respect to the {\bf translator metric} $g_{ij}(x,y,z)=e^{-z}\delta_{ij}$.
(See~\cite{himw-survey} for an expository survey of some recent results about translators.)
In this paper, we prove the following theorem:

\begin{theorem}\label{main-theorem}
For every $a>0$, there is a unique translator $M_a$ with the following properties:
\begin{enumerate}
\item\label{smooth-item} $M_a$ is a smooth, properly embedded surface in $\RR^3$.
\item\label{vertical-line-item} For each integer $n$, $M$ contains the vertical line $\{(na,0)\}\times\RR$.
\item\label{periodic-item} $M_a$ is periodic with period $(2a,0,0)$.
\item\label{graph-item} $M_a\cap\{y>0\}$ is the graph of a function $u_a$ defined
on some strip $\RR\times (0,b)$, with boundary values given by
\begin{align*}
u_a(x,0) &= -\infty \quad\text{for $-a<x<0$}, \\
u_a(x,0) &= +\infty \quad\text{for $0<x<a$}, \\
u(x,b) &= -\infty \quad\text{for all $x$}.
\end{align*}
\item\label{origin-item} $M_a$ is tangent to the $yz$-plane at the origin.
\end{enumerate}
If $M'$ is any other translator
  with properties~\eqref{smooth-item}--\eqref{graph-item}, then $M'$ is a vertical translate of $M_a$.

Furthermore, the width $b=b(a)$ of the strip in~\eqref{graph-item}
is a continuous, increasing function 
taking values in $(\pi/2, \pi)$ and tending to $\pi/2$ as $a\to 0$ and to $\pi$ as $a\to\infty$.
The surface $M_a$ depends smoothly on $a$.
As $a\to 0$, $M_a$ converges smoothly away from the $x$-axis $X$ to the union of the $xz$-plane
and the grim reaper surface $\{(x,y,z): \text{$z = \log(\cos y)$ and $|y|<\pi/2$}\}$.
\end{theorem}
\begin{figure}[htbp]
\begin{center}
\includegraphics[height=.3\textheight]{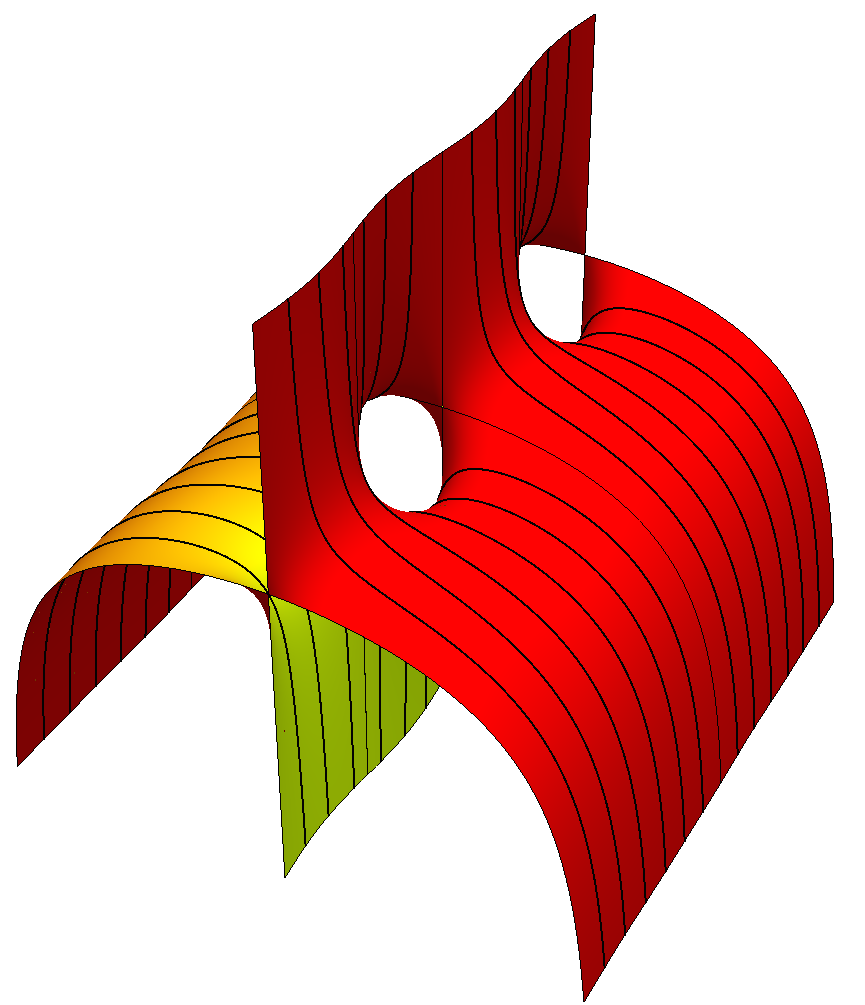}
\end{center}
\caption{The surface $M_1$ .} \label{fig:one}
\end{figure}

The behavior of $M_a$ as $a\to\infty$ is described in~\S\ref{to-infinity-section}.

We do not know whether the function $b(\cdot)$ is strictly increasing.

We remark that the periodicity in Property~\eqref{periodic-item} follows from Property~\eqref{vertical-line-item}
since any translator must 
be  invariant under rotation through $\pi$ about each vertical line that it contains.
(This rotational invariance of the translator is an instance of the Schwarz reflection principle for minimal surfaces, since the translator metric is rotationally invariant about vertical lines.  In particular, the surface and the rotated surface are tangent along the entire line and therefore must coincide, since distinct minimal surfaces in $3$-manifolds can have at most isolated points of tangency.)
Also, from the uniqueness of $M_a$, we see that $M_a$ must be invariant under reflection in the plane $x=a/2$
(or, more generally, in the plane $x=ka/2$ if $k$ is an odd integer.)

For all sufficiently small $a>0$, Xuan Hien Nguyen~\cite{nguyen-tridents} used desingularization methods
 to prove  existence of translators with
 properties~\eqref{smooth-item}--\eqref{origin-item}.  She called such surfaces {\bf tridents}.
We were led to Theorem~\ref{main-theorem} by a desire to understand tridents
  from a variational point of view.
According to the uniqueness assertion in Theorem~\ref{main-theorem}, our examples coincide with hers in the range
of $a$'s for which she proves existence.
\begin{figure}[htbp]
\begin{center}
\includegraphics[height=.25\textheight]{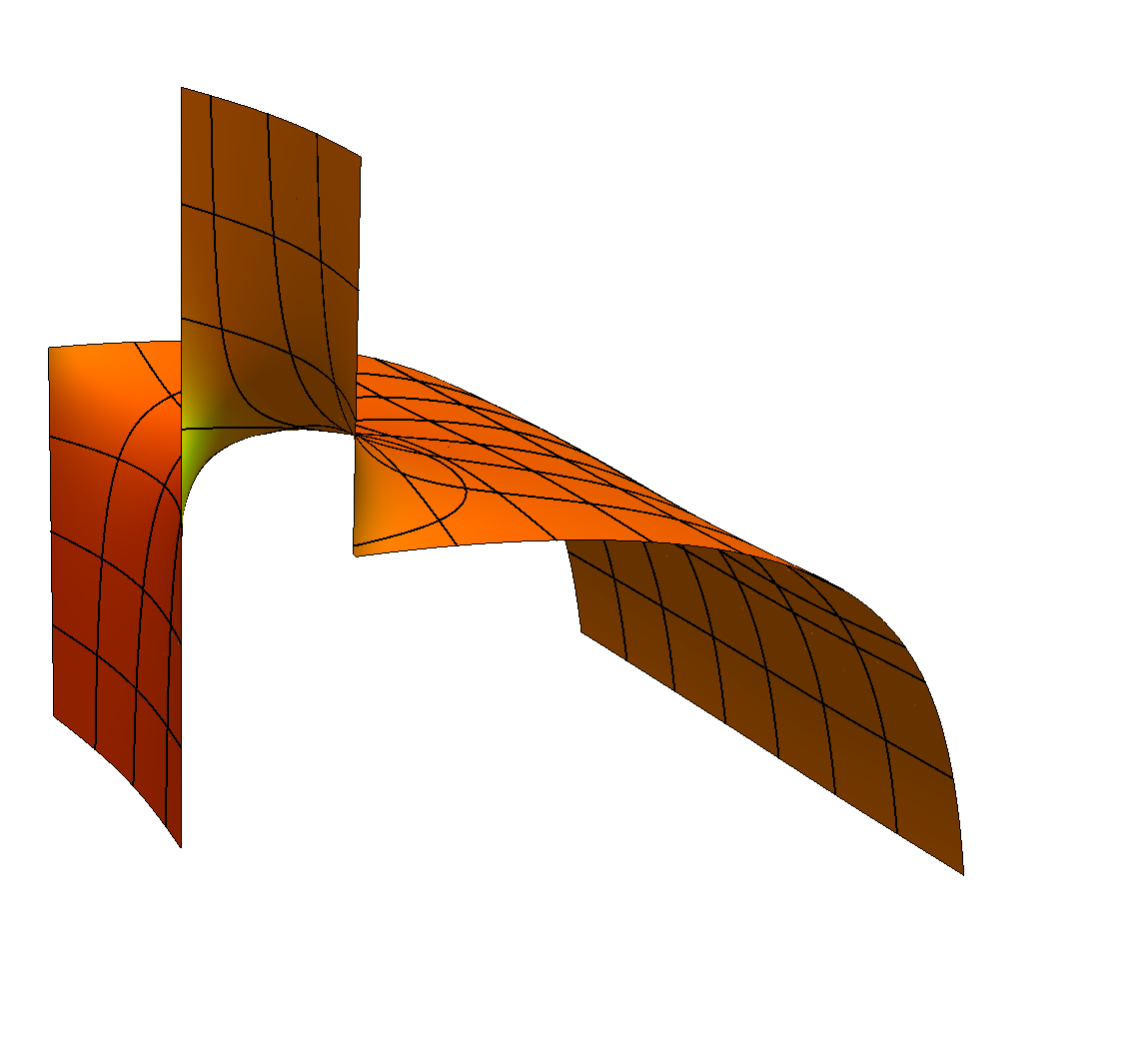}
\end{center}
\caption{A fundamental piece of the surface $M_a$. It is a graph over the domain $(-a,a) \times (0,b(a))$.
The whole surface is obtained by successive Schwarz reflections about the vertical lines.} \label{fig:fund}
\end{figure}

Each trident is an example of a {\bf semigraphical translator}, i.e., a properly embedded 
translator $M$ (without boundary) that contains
a discrete, nonempty collection $L$ of vertical lines such that $M\setminus L$ is a graph.
In~\cite{scherkon}, we constructed other families of semigraphical translators.   In~\S\ref{semigraphical-section}, 
we almost classify
all semigraphical translators. In particular, we show that every semigraphical translator, with
the possible exception of one type that we believe cannot occur,
is either a trident or one of the examples in~\cite{scherkon}.

(All graphical translators were classified in~\cite{himw}.)

Although tridents have finite entropy (the entropy is $3$), they have infinite genus and thus
cannot arise as blowups of an initially smooth mean curvature flow.
However, any subsequential limit of tridents as $a\to \infty$ has finite entropy
and finite topology (the topology of a disk).  Whether it could occur as a blowup is an interesting
open problem.

The organization of the paper is as follows.  In~\S\ref{existence-section}, we prove existence of the examples $M_a$
and that $b(a)$ is an increasing function of $a$.
In~\S\ref{uniqueness-section}, we prove uniqueness.  
In~\S\ref{property-section}, we prove a geometric property of tridents.
In~\S\ref{dependence-section},
we prove that $M_a$ depends smoothly on $a$ and that the width $b(a)$ is a continuous
function of $a$.
In~\S\ref{to-zero-section} and~\S\ref{to-infinity-section}, we analyze the behavior of $M_a$ as $a$ tends
to $0$ and to infinity.
In~\S\ref{widths-section}, we show that the set $\{b(a):a>0\}$ of widths of the various examples
is the open interval $(\pi/2,\pi)$.
In~\S\ref{semigraphical-section}, we discuss the classification of semigraphical translators.

\section{Existence}\label{existence-section}

\newcommand{\length}{\operatorname{length}}

\begin{lemma}\label{prelim-lemma}
Suppose that $K$ is a $2$-manifold with Riemannian metric $\gamma$.
Endow $K\times\RR$ with the translator metric
\[
     g(p,z)=e^{-z}(\gamma(p) +\,dz^2).
\]
\begin{enumerate}[\upshape (1)]
\item\label{area-identity-item}
If $M$ is a surface of finite area in $K\times\RR$ such that
$M$ is minimal or such that all the tangent planes to $M$ are vertical, then
\[
  \area(M) = \int_{\partial M}  \vv \cdot \nn \,ds,
\]
where $\vv = -\partial/\partial z$ and $\nn$ is the unit vector tangent to $M$ and normal to $\partial M$
that points out from $M$.
\item\label{second-prelim-item}
Suppose that $C$ is the union of one or more curves in $K$, and that
 $M$ is a minimal surface in $K\times[0,\infty)$ of finite area with $\partial M = \partial (C\times[0,\infty))$.
Then
\[
    \area(M)\le \area(C\times [0,\infty)) = \length(C),
\]
with equality if and only if $M$ is tangent to $C\times [0,\infty)$ along $C\times \{0\}$.
If $C$ is a union of geodesics, then equality holds if and only if $M=C\times[0,\infty]$.
\end{enumerate}
\end{lemma}

\begin{proof}
Let 
\begin{align*}
&\phi_t: K\times\RR\to K\times \RR, \\
&\phi_t(p,z) = (p,z-t).
\end{align*}
Thus
\[  
\pdf{}t \phi_t(p,z) = \vv.
\]
Now
\[
   \area(\phi_t(M)) = e^t \area(M),
\]
so 
\[
  (d/dt)_{t=0}\area(\phi_t(M)) = \area(M).
\]
Thus by the first variation formula,
\[
 \area(M) = -\int_M H\cdot \vv + \int_{\partial M} \vv\cdot \nn\,ds,
\]
where $H$ is the mean curvature.  If $M$ is minimal or if
the tangent planes to $M$ are vertical, then $H\cdot \vv\equiv 0$, so
\begin{equation}\label{area-identity}
  \area(M)=  \int_{\partial M} \vv\cdot \nn\,ds.
\end{equation}
Thus we have proved Assertion~\eqref{area-identity-item}.

Now suppose $M$ is either $C\times [0,\infty)$ or a minimal surface
with boundary $\partial (C\times [0,\infty))$.  Then by~\eqref{area-identity},
\begin{align*}
\area(M)
&=
\int_{\partial M} \vv\cdot\nn\,ds \\
&=
\int_{C\times\{0\}} \vv\cdot\nn\,ds + \int_{(\partial C)\times [0,\infty)} \vv\cdot\nn\,ds \\
&=
\int_{C\times\{0\}} \vv\cdot\nn\,ds
\end{align*}
since $\vv\cdot \nn\equiv 0$ along $(\partial C)\times [0,\infty)$.

When $z=0$, $\vv$ is a unit vector so $\vv\cdot\nn\le 1$.  Thus
\[
\area(M) \le \length(C)
\]
with equality if and only if $M$ is tangent to $C\times [0,\infty)$ along $C\times\{0\}$.

If $C$ consists of geodesics, then $C\times [0,\infty)$ is minimal, so if $M$ and $C\times[0,\infty)$
are tangent along $C\times \{0\}$, then they coincide by unique continuation.
\end{proof}

Now we specialize as follows.  
Let  $K=K_{a,b}$ be the rectangle
\[
   [-a,a]\times [0,b] = \{(x,y): -a\le x \le a, \, 0\le y\le b\}
\]
with the right and left edges identified.
(Equivalently, $K$ is the quotient of the strip $\RR\times [0,b]$ under the equivalence $(x,y)\equiv (x+2a,y)$.)
Note that $y$ is well-defined on $K$, and that $x$ is well-defined except that it is two-valued
on the segment $\{x=a\}$ (which is the same as the segment $\{x=-a\}$.)

\begin{lemma}\label{geodesics-lemma}
Let $\Gg$ be the collection of embedded geodesics $\sigma$ in $K$ such that
$\partial \sigma\subset \{(0,0), (a,0)\}$.  Then $\Gg$ consists of
\begin{enumerate}
\item the segment $\{(x,0): 0< x < a\}$,
\item the segment $\{(x,0): -a < x < 0\}$, and
\item the closed geodesics $\{(x,y): y=\beta\}$ where $0\le \beta\le b$.
\end{enumerate}
Thus if $\Omega$ is a connected open set in $K\setminus \partial K$ 
whose boundary consists of curves in $\Gg$,
then $\Omega$ is  $\{(x,y)\in  K: \alpha < y < \beta\}$ for some $\alpha$ and $\beta$ with $0\le\alpha<\beta\le b$.
\end{lemma}

The proof is trivial.

Now consider $K_{a,b} \times\RR$ endowed with the translator metric.
Let
\begin{align*}
   P&=P_{a,b} = \{(x,0)\in K: 0< x < a\}, \\
   N &= N_{a,b} = \{(x,0)\in K: -a<x<0\} \cup \{(x,y)\in K: y=b\},  \\
   \Gamma &=\Gamma_{a,b} = \partial (N\times[0,\infty)).
 \end{align*}
 \begin{figure}[htbp]
\begin{center}
\includegraphics[width=.55\textwidth]{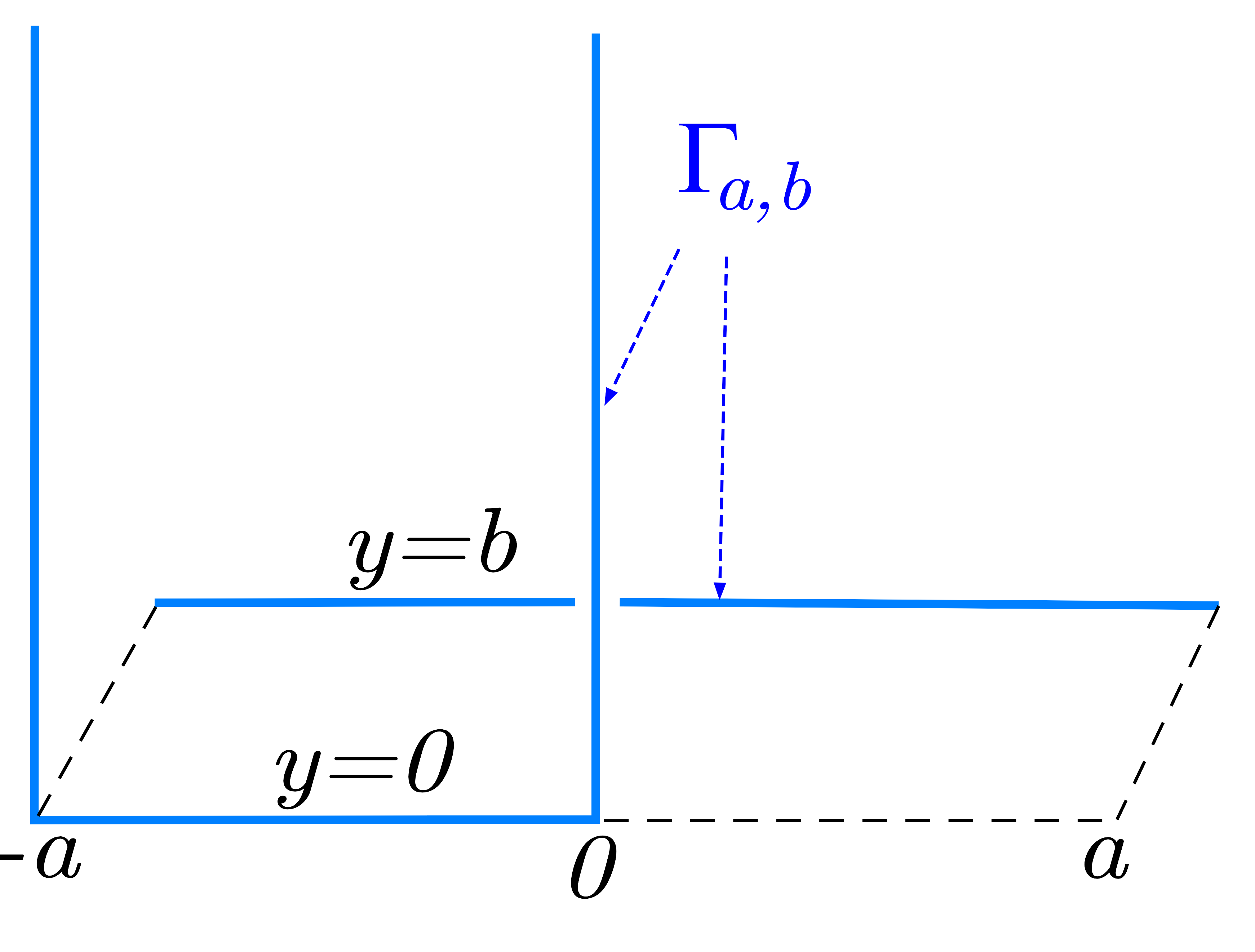}
\caption{The curve $\Gamma_{a,b}.$}
\label{fig:gamma}
\end{center}
\end{figure}

\begin{definition}\label{alpha-definition}
We let $\alpha(a,b)$ be the infimum of $\area(M)$ among surfaces $M$ in $K_{a,b}\times[0,\infty)$
having boundary $\Gamma_{a,b}$.
\end{definition}

By Assertion~\eqref{second-prelim-item} of Lemma~\ref{prelim-lemma} (or by direct calculation),
\[
   \area(N_{a,b}\times [0,\infty)) = \length(N_{a,b})=3a,
\]
so
\begin{equation}\label{alpha-3a}
\begin{gathered}
  \text{$\alpha(a,b) \le 3a$, with equality if and only}
  \\
  \text{if $N_{a,b}\times [0,\infty)$ is area-minimizing.}
\end{gathered}
\end{equation}

\begin{proposition}\label{first-existence-proposition}
Let $\Gamma=\Gamma_{a,b}$, $N=N_{a,b}$, and $K=K_{a,b}$ be as above.
\begin{enumerate}[\upshape (i)]
\item\label{other-minimal-item}
 If $\Gamma$ bounds any finite area minimal surface not equal to $N\times[0,\infty)$, then 
  $\alpha(a,b) < 3a$.
\item\label{existence-item}
 If $\alpha(a,b)< 3a$, then there is a unique least-area surface $M$ with boundary $\Gamma$,
and $M\setminus \partial M$ is the graph of a function 
\[
    u=u_{a,b}: K\setminus \partial K \to \RR.
\]
\end{enumerate}
\end{proposition}

\begin{proof}
Assertion~\eqref{other-minimal-item} follows immediately from~\eqref{alpha-3a} and 
 Assertion~\eqref{second-prelim-item} of Lemma~\ref{prelim-lemma}.

\begin{claim}\label{connected-claim}
If $M$ is a minimal surface in $K\times[0,\infty)$ with boundary $\Gamma$ and if $M$ is not connected,
then $M=N\times [0,\infty)$.
\end{claim}

For suppose that $M$ is such a surface. Then one component $M_1$ of $M$ 
has boundary in $\{y=0\}$ and the other component $M_2$ has boundary in $\{y=b\}$.
Using grim reaper surfaces as  barriers on the lift of $M$ to $\RR^3$, one sees that $M_1$ lies
in $\{y=0\}$ and that $M_2$ lies in $\{y=b\}$.
Hence $M=N\times [0,\infty)$, so Claim~\ref{connected-claim} is proved.

Now suppose $\alpha(a,b)< 3a = \area(N\times[0,\infty))$ and let $M$ be a least-area surface with
 boundary $\Gamma$.
Then $M\ne (N\times [0,\infty))$, so $M$ is connected.

By the strong maximum principle, $M\setminus \partial M$ lies in the interior of $K\times [0,\infty)$.

We assert that $M\setminus \partial M$ is the graph of a function $u:\Omega\to (0,\infty)$ over
some open subset $\Omega$ of $(K\setminus \partial K)$.
For if not, we could find a $\lambda>0$ such that $\Sigma=(M\setminus \partial M)$ and $\Sigma+\lambda\ee_3$
intersect each other.  By changing $\lambda$ slightly, we can assume that $\Sigma$ and $\Sigma+\lambda\ee_3$
intersect transversely.   
But that is impossible by a standard cut-and-paste argument (Lemma~\ref{cut-paste} below).

Thus $M\setminus \partial M$ is the graph of a function $u:\Omega\to\RR$.

Note that as $\lambda\to\infty$, 
the curve $\Gamma - \lambda \ee_3$ converges smoothly to $(\partial P)\times \RR$
and the surface
$M - \lambda\ee_3$ converges smoothly to $(\partial_\infty\Omega) \times \RR$, where 
\[
  \partial_\infty\Omega = \{p\in \partial \Omega: u(p)=\infty\}.
\]
(The convergence is smooth by the curvature estimate in Theorem~\ref{main-curvature-estimate} below.)
Now $(\partial_\infty\Omega)\times\RR$ is minimal (since it is the limit of the minimal surfaces
$M-\lambda\ee_3$), so $\partial_\infty\Omega$ is a union of geodesics.  Since the boundary of
   $(\partial_\infty\Omega)\times\RR$ is $(\partial P)\times \RR$, the geodesics belong
   to the family $\Gg$ in Lemma~\ref{geodesics-lemma}.
Likewise $\partial_{-\infty}\Omega$ consists of geodesics in $\Gg$.  
Thus by Lemma~\ref{geodesics-lemma},
\[
  \Omega =\{(x,y): \alpha< y < \beta\}
\]
for some $0\le \alpha<\beta\le b$.  Since $\partial M=\Gamma$, we see that $\Omega$
must be all of $K\setminus \partial K$.

It remains only to show that $M$ is unique.   Suppose that there is another least-area surface $M'$
with boundary $\Gamma$.  Then we could find a $\lambda\in \RR$ such that 
$M\setminus \Gamma$ and $(M'\setminus \Gamma)+ \lambda \ee_3$ intersect transversely.
By relabelling, we can assume that $\lambda\ge 0$.
But that is impossible by the cut-and-paste principle (Lemma~\ref{cut-paste}).
\end{proof}

\begin{corollary}\label{pi-corollary} $\alpha(a,\pi)=3a$.
\end{corollary}

\begin{proof}
Suppose not. Then $\alpha(a,\pi)<3a$, so by Proposition~\ref{first-existence-proposition},
 $\Gamma_{a,\pi}$ would bound
a connected minimal surface $M$ in $K_{a,\pi}\times [0,\infty)$.
Consider the grim reaper surface 
\[
  z = f(x,y):=\log(\sin y)
\]
Note that 
\[
  (x,y,z)\in M\setminus \partial M \mapsto f(x,y)-z
\]
would attain a maximum, contradicting the strong maximum principle.
\end{proof}

\begin{proposition}\label{alpha-bounds}
Let $\alpha(a,b)$ be as in Definition~\ref{alpha-definition}.
If $a<a'$, then
\begin{equation}\label{a-a-prime}
  \alpha(a,b) \le \alpha(a',b) \le \frac{a'}a \alpha(a,b).
\end{equation}
and if $b< b'$, then
\begin{equation}\label{b-b-prime}
  \alpha(a,b)\le \alpha(a,b') \le \frac{b'}b \alpha(a,b).
\end{equation}
Thus 
\[
  (a,b)\mapsto \frac{\alpha(a,b)}{a}
\]
is a continuous function that is decreasing as a function of $a$ and increasing as a function of $b$.
\end{proposition}

\begin{proof}
Let
\begin{align*}
&F: K_{a,b}\times \RR \to K_{a',b}\times \RR, \\
&F(x,y,z) = F((a'/a)x,y,z).
\end{align*}
Then 
\[
   M \mapsto F(M) \tag{*}
\]
is a bijection from surfaces in $K_{a,b}\times [0,\infty)$ with boundary $\Gamma_{a,b}$
to surfaces in $K_{a',b}\times [0,\infty)$ with boundary $\Gamma_{a',b}$.
Note that at each point of $M$, the Jacobian determinant of~\thetag{*} is between $1$ and $a'/a$,
so
\[
    \area(M) \le \area(F(M)) \le \frac{a'}{a} \area(M),
\]
which implies the inequality~\eqref{a-a-prime}.  The inequality~\eqref{b-b-prime}
 is proved in exactly the same way using the map 
$(x,y,z)\mapsto (x, (b'/b)y,z)$.
\end{proof}

\begin{theorem}[Main Existence Theorem]\label{existence-theorem}
For each $a>0$, there is a unique $b(a)\in (0,\pi]$ such that 
\begin{equation}\label{cut-off}
\begin{aligned}
\alpha(a,b) < 3a \quad&\text{if $b< b(a)$, and} \\
\alpha(a,b) = 3a \quad&\text{if $b\ge b(a)$.}
\end{aligned}
\end{equation}
Furthermore, $b(a)$ is an increasing function of $a$.
If $b=b(a)$, there is a smooth minimal surface $M=M_a$ in $K_{a,b}\times \RR$
such that
\begin{enumerate}[\upshape (i)]
\item The boundary of $M$ consists of the lines $\{(0,0)\}\times \RR$ and $\{(a,0)\}\times \RR$,
\item $M$ as smoothly asymptotic as $z\to\infty$ to $P\times \RR$
and as $z\to -\infty$ to $N\times \RR$.
\item $M\setminus \partial M$ is the graph of a function
\[
    u: \{(x,y)\in K: 0<y<b\} \to \RR
\]
such that $u=\infty$ on $P$ and $u=-\infty$ on $N$.
\end{enumerate}
\end{theorem}

\begin{proof}
Consider the surface 
\[
   \Sigma = (K\times \{0\}) \cup P\times[0,\infty).
\]
The boundary of this surface
is $\Gamma_{a,b}$, and its area is
\[
   2ab + a = (2b+1)a.
\]
In particular,
\[  
  \area(\Sigma_{a,b}) < 3a \quad\text{if $b<1$}.
\]
By Corollary~\ref{pi-corollary}, 
\[
 \frac{\alpha(a,b)}{a} = 3 \quad\text{if $b\ge \pi$}.
\]
Since $\alpha(a,b)$ is a continuous, increasing function of $b$,
there is a unique $b(a)$ such that~\eqref{cut-off} holds.

Since $\alpha(a,b)/a$ is a decreasing function of $a$ (by Proposition~\ref{alpha-bounds}), we see that $b(a)$
is an increasing function of $a$.

Now fix an $a>0$ and let $b_i\in (0,b(a))$ be a sequence converging to $b(a)$.

Let $M_i$ be the area-minimizing surface in $K_{a,b_i}\times[0,\infty)$ with
boundary $\Gamma_{a,b_i}$.  Since $b_i<b(a)$, $M_i\setminus \partial M_i$ 
is a graph.

For $z\ge 0$, let $(\cos\theta_i(z),\sin\theta_i(z),0)$ be tangent to $M_i$ at $(0,0,z)$.
Note that $\theta_i(\cdot)$ is monotonic since $M_i\setminus\partial M_i$ is a graph.
Since $\theta_i(0)=\pi$ and $\theta_i(z)\to 0$ as $z\to\infty$, we see that $\theta_i(\cdot)$
is a decreasing function that takes all values in $[0,\pi)$.
In particular, there is a $z_i\in (0,\infty)$ for which $\theta_i(z_i)=\pi/2$.

Recall (see~\eqref{alpha-3a}) that $N_{a,b(a)}\times [0,\infty)$ is the only minimal surface with boundary
$\Gamma_{a,b(a)}$.  Thus $M_i$ converges smoothly to $N_{a,b(a)}\times [0,\infty)$ as $b_i\to b(a)$.
Hence $\theta_i(z)\to \pi$ for every $z\in (0,\infty)$, so
\[
    \lim_{i\to\infty} z_i = \infty.
\]
By passing to a subsequence, we can assume
 (by the curvature estimate in Theorem~\eqref{main-curvature-estimate} below)
  that the surfaces $M_i- (0,0,z_i)$
converge smoothly to a limit surface whose boundary consists 
of the vertical lines through $(0,0)$ and $(a,0)$.
Let $M$ be the component of the limit surface that contains the  origin.

Since $M$ is a limit of graphs, either $M\setminus \partial M$ is a graph or else $M$ is flat and vertical.
If it were flat and vertical, it would be $\sigma\times\RR$ for some geodesic
$\sigma$ in the family $\Gg$ in Lemma~\ref{geodesics-lemma}.   But there is no such geodesic
since $\ee_2$ is tangent to $M$ at the origin.  Thus $M\setminus \partial M$ is the graph
of a function 
\[
   u: \Omega\to \RR
\]
for some open set containing $(0,0)$ in its boundary.
Exactly as in the proof of Proposition~\ref{first-existence-proposition},
    $\Omega$ must be $\{(x,y): 0< y<\beta\}$ for some $\beta\le b(a)$.

Since the angle functions $\theta_i(z)$ are monotonically decreasing, the corresponding angle
function $\theta(z)$ for $M$ must also be decreasing.   It follows that
\begin{align*}
   u(x,0) &= \infty \qquad\text{for $0<x<a$, and}
   \\
   u(x,0) &= -\infty\qquad \text{for $-a<x<0$}.
\end{align*}

Using grim reaper surfaces as barriers, one sees that $u$ must be $-\infty$ (not $\infty$) on $\{y=\beta\}$.

Note that as $\lambda\to\infty$, $M - (0,0,\lambda)$ converges smoothly
to $N\times \RR$.  Thus $N\times [0,\infty)$ is area-minimizing.
  Hence $\beta\ge b(a)$ by~\eqref{alpha-3a}, and therefore $\beta=b(a)$.
\end{proof}

In the proof of Theorem~\ref{existence-theorem}, we used the following standard
cut-and-paste principle:

\begin{lemma}[Cut-and-Paste Lemma]\label{cut-paste}
Let $\Nn$ be an $(m+1)$-dimensional Riemmannian manifold with trivial $m^\textnormal{th}$ homology and
with mean-convex, piecewise smooth boundary.
Let $A_1 \subset A_2\subset \partial \Nn$ be regions with finite area.  Let $B_i$ be an area-minimizing
flat chain mod $2$ in $\Nn$ having the same mod $2$ boundary as $A_i$.  
Suppose that no connected component of $\spt(B_i)\setminus \spt(\partial B_i)$ lies in $\partial \Nn$.
Then the regular sets of $B_1$ and $B_2$
cannot intersect transversely at any interior point.
\end{lemma}

\begin{proof}
By the maximum principle, $\spt B_i\setminus \spt \partial B_i$ lies in the interior of $\Nn$.
Let $\Omega_i$ be the region bounded by $A_i$ and $\spt B_i$.
Let $B'_1$ be the portion of $\partial (\Omega_1\cap \Omega_2)$ in the interior of $N$
and let $B'_2$ be the portion of $\partial (\Omega_1\cup \Omega_2)$ in the interior of $N$.
Then $B_i'$ and $B_i$ have the same mod $2$ boundary, so
\[
   \area(B_i)\le \area(B_i').
\]
Also, 
\[
   \area(B_1')+\area(B_2') = \area(B_1') + \area(B_2').
\]
Thus $\area(B_i')=\area(B_i)$ for $i=1,2$, so $B_1'$ and $B_2'$ are area-minimzing.
But at a point where the regular sets of $B_1$ and $B_2$ cross transversely,
the tangent cone to $B_i'$ is a pair of a halfplanes that meet at an angle $\theta\in (0,\pi)$
along their common edge.  But that contradicts the fact that $B_i'$ is area-minimizing.
\end{proof}

In the proof of Theorem~\ref{existence-theorem},
we used the following curvature estimate, which is Theorem A.3 in~\cite{scherkon}:

\begin{theorem}\label{main-curvature-estimate}
There is a constant $C<\infty$ with the following property.
Let $M$ be translator with velocity $-\ee_3$ in $\RR^3$ such that
\begin{enumerate}
\item $M$ is the graph of a smooth function $u:\Omega\to\RR$ on a convex open subset $\Omega$ of $\RR^2$.
\item $\Gamma:=\overline{M}\setminus M$ is a polygonal curve (not necessarily connected)
       consisting of segments, rays, and lines.
\item $\overline{M}$ is a smooth manifold-with-boundary except at the corners of $\Gamma$.
\end{enumerate}
If $p\in \RR^3$, let $r(M,p)$ be the supremum of $r>0$ such that $\BB(p,r)\cap \partial M$ is either empty
or consists of a single line segment, where $\BB(p,r)$ is the Euclidean ball of radius $r$.
Then
\[
   |A(M,p)| \min \{1, r(M,p) \} \le C,
\]
where $|A(M,p)|$ is the norm of the second fundamental form of $M$ at $p$ with respect to the Euclidean metric.
\end{theorem}

\section{Uniqueness}\label{uniqueness-section}

\begin{theorem}[Uniqueness Theorem]\label{uniqueness-theorem}
Suppose that
\[
   u, v: K_{a,b}\setminus \partial K_{a,b} \to \RR
\]
are solutions of the translator equation such that
\begin{align*}
u(x,0) &= -\infty \quad\text{for $-a<x<0$,} \\
u(x,0) &= \infty \quad\text{for $0<x<a$,} \\
u(x,b) &\equiv -\infty,
\end{align*}
and such that 
\begin{align*}
   v(x,0) &= -\infty \quad \text{for $-a<x<0$}, \\
   v(x,0) &=  \infty \quad \text{for $0<x<a$}.
\end{align*}
Then $u - v$ is constant.
\end{theorem}

Note that there is no assumption about behavior of $v$ on the portion of
the boundary where $y=b$.

\begin{corollary}\label{invariance-corollary}
Let $u$ be as in Theorem~\ref{uniqueness-theorem}.
Then $u$ is invariant under $(x,y)\mapsto (a/2-x,y)$ and under
$(x,y)\mapsto (-a/2-x,y)$.
\end{corollary}

Note that the Uniqueness Theorem~\ref{uniqueness-theorem} implies the 
uniqueness of the surface $M_a$ in Theorem~\ref{main-theorem}, and thus
that $M_a$ has the reflectional invariance given by Corollary~\ref{invariance-corollary}.

\begin{proof}[Proof of Theorem~\ref{uniqueness-theorem}]
By adding a large positive constant to $u$, we can assume that:
\begin{enumerate}
\item the zero set $\{u=0\}$
of $u$ is the union 
\[
   \graph(\gamma_1) \cup \graph(\gamma_2)
\]
where 
\begin{align*}
&\gamma_1: (-a,0)\to (0,b) \\
&\gamma_2: [-a,a]\to (0,b)
\end{align*}
are smooth functions such that
\[
    \gamma_1(x)< \gamma_2(x) \quad\text{for $-a<x<0$}
\]
and such that
\[
    \lim_{x\to -a}\gamma_1(x) = \lim_{x\to 0}\gamma_1(x) = 0,
\]
\item
\begin{align*}
  \pdf{u}y &> 0 \quad\text{on $\Omega_1:=\{(x,y): 0<y<\gamma_1(x)\}$}, \\
  \pdf{u}y &< 0 \quad\text{on $\Omega_2:=\{(x,y): \gamma_2(x)<y<b\}$}.
\end{align*}
\end{enumerate}

Thus
\[
   \{u<0\}  = \Omega_1\cup \Omega_2.
\]

Note that $\gamma_1$ extends to a smooth function on $[0,a]$
by setting $\gamma_1(0)=\gamma_2(a)=0$, and that $\gamma_1'(0)$ and $\gamma_2'(a)$
are nonzero.

Now $v$ is bounded on each level set of  $u$.
(See Remark~\ref{bounded-remark} below.)
Thus by adding a constant to $v$, we can assume that
\begin{equation}\label{touching}
  \inf \{ v(x,y): u(x,y)=0\} = 0.
\end{equation}
Since $u$ and $v$ cannot have interior local minima, it follows that
\[
     \text{$u, v\ge 0$ on $(\Omega_1\cup\Omega_2)^c$}.
\]
In particular, the set $\{v\le 0\}$ is contained in $\{u\le 0\}$.

By Lemmas~\ref{tweedledee} and~\ref{tweedledum} below,
\begin{equation}\label{omegas-good}
    v\ge u \quad\text{on $\Omega_1\cup \Omega_2$}.
\end{equation}

Since $u$ and $v$ are bounded below on $(\Omega_1\cup \Omega_2)^c$ and since $v\ge u$
on the boundary of $(\Omega_1\cup\Omega_2)^c$, it follows (see Lemma~\ref{topsy} below)
that $v \ge u$ on $(\Omega_1\cup \Omega_2)^c$.
Combining this with~\eqref{omegas-good}, we see that $v\ge u$ everywhere.
On the other hand, the graphs of $u$ and $v$ touch at some point on $\overline{\{u=0\}}$
by~\eqref{touching}.  Thus $u\equiv v$ by the strong maximum principle or boundary maximum principle.
\end{proof}

\begin{remark}\label{bounded-remark}
Here we explain why $v$ is bounded on each level set of $u$.
Standard boundary regularity results (e.g.,~\cite{hardt-simon}) imply that $\graph(u)$ and
 $\graph(v)$ extend
smoothly (as a manifolds-with-boundary) to the lines $\{(0,0)\}\times\RR$ and $\{(a,0)\}\times\RR$.
Hence there are diffeomorphisms  
\[
  \theta^u, \theta^v : \RR\to (0,\pi) 
\]
such that $(\cos\theta^u(z),\sin\theta^u(z),0)$ and $(\cos\theta^v(z),\sin\theta^v(z),0)$
are tangent to $\graph(u)$ and $\graph(v)$ at $(0,0,z)$.  
Consider points $p_i=(r_i\cos\theta_i, r_i\sin\theta_i)$ converging to $(0,0)$.
Note that $u(p_i)\to c$ if and only if $\theta_i\to \theta^u(c)$ and $v(p_i)\to c$ if and only
if $\theta_i\to \theta^v(c)$.
The analogous statements hold at $(a,0)$.
Boundedness of $v$ on each level set of $u$ follows immediately.
\end{remark}

\begin{lemma}\label{tweedledee}
Let $I$ be a finite open interval and $\gamma: \overline{I}\to \RR$ be a continuous function that is positive on $I$
and that vanishes on the endpoints of $I$.  Let $\Gamma$ be the graph of $\gamma|I$:
\[
 \Gamma   = \{(x,\gamma(x)) : x\in I \},
\]
and let $\Omega$ be the open region between $I\times \{0\}$ and $\Gamma$:
\[
   \Omega = \{(x,y)\in I\times \RR: 0< y < \gamma(x)\}.
\]
Suppose that
\[
  u, v: \Omega\to \RR
\]
are solutions to the translator equation that extend continuously to $\Omega\cup\Gamma$.
Suppose also that
\begin{align*}
\pdf{u}{y}>0 \quad&\text{on $\Omega$}, \\
u \le v \quad&\text{on $\Gamma$, and} \\
u(x,0) = -\infty \quad&\text{for $x\in I$.}
\end{align*}
Then $u\le v$ on $\Omega$.
\end{lemma}

\begin{proof}
Suppose not.  Then there would be a maximum $s>0$ such that
\[
    \graph(u) + s\ee_2
\]
intersects $\graph(v)$.  In a neighborhood of a point $p$ of contact,
$\graph(v)$ would lie in the closed region above $\graph(u)+s\ee_2$,
and the two surfaces would be tangent at $p$.  But then by the strong maximum
principle and unique continuation, $\graph(u)+s\ee_2$
and $\graph(v)$ would coincide, which is clearly impossible.
\end{proof}

Lemma~\ref{tweedledee} has an analog for periodic functions: 

\begin{lemma}\label{tweedledum}
Let $\gamma: \RR \to \RR$ be a continuous function that is 
everywhere $>0$ and that is periodic with period $L$.
 Let $\Gamma=\{(x,\gamma(x): x\in \RR\}$
be the graph of $\gamma$, 
and let
\[
   \Omega = \{(x,y): 0< y < \gamma(x)\}.
\]
Suppose that
\[
  u, v: \Omega\to \RR
\]
are solutions to the translator equation that are periodic with period $(L,0)$ and 
that extend continuously to $\Omega\cup\Gamma$.
Suppose also that
\begin{gather*}
\pdf{u}{y}>0 \quad\text{on $\Omega$}, \\
u \le v \quad\text{on $\Gamma$, and} \\
u(x,0) = -\infty \quad\text{for all $x$.}
\end{gather*}
Then $u\le v$ on $\Omega$.
\end{lemma}

The proof is almost identical to the proof of Lemma~\ref{tweedledee}.

\begin{lemma}\label{topsy}
Suppose that  $\Omega$ be a domain in a Riemannian $2$-manifold 
with piecewise smooth boundary of finite length.
Let $u,v:\Omega\to\RR$ be solutions of the translator equation that are bounded below.
If $u\le v$ on $\partial \Omega$, then $u\le v$ on $\Omega$.
\end{lemma}

\begin{proof}
By translating, we can assume that $u$ and $v$ are bounded below by $0$.
Note that for any compact region $\overline{W}$ in $\overline{\Omega}\times\RR$,
\[
    \area(\graph(u)\cap W) \le \area((\partial W)\cap \{z>u\}).
\]
Applying this to $W_n=\overline{\Omega_n}\times [0,n]$
where $\Omega_n$ is a nice exhaustion of $\Omega$, and taking the limit, we
see that
\begin{align*}
\area(\graph(u)) 
&\le \area(\partial \Omega\times [0,\infty)) \\
&= \length(\partial \Omega)  \\
&<\infty.
\end{align*}

Likewise, $\area(\graph(v))<\infty$.
Hence $u\le v$ by the cut-and-paste argument (Lemma~\ref{cut-paste}).
\end{proof}

\section{A Geometric Property of Tridents}\label{property-section}

\begin{theorem}\label{property-theorem}
Let $M_a$ and $u_a: \RR\times (0,b(a))\to \RR$ be as in Theorem~\ref{existence-theorem}.
Then
\begin{align*}
   \pdf{}x u_a(x,y)&>0 \quad\text{for $-a/2<x<a/2$}, \\
   \frac{\partial^2}{\partial x^2} u_a(-a/2,y) &> 0, \\
   \frac{\partial^2}{\partial x^2}u_a(a/2,y) &< 0
\end{align*}
for $0<y<b(a)$.
\end{theorem}

\begin{lemma}\label{property-lemma}
If $a>0$, if $0<b<\pi$, and if $f:\RR\to\RR$ is a smooth, $2a$-periodic function,
then there is a unique $2a$-periodic solution
\[
    u: \RR\times (0,b)\to\RR
\]
to the translator equation with boundary values $u(\cdot,b)\equiv 0$ and $u(\cdot,0)=f(\cdot)$.
Furthermore, if $f$ is invariant under $x\mapsto a/2-x$ and under $x\mapsto -a/2-x$ and if 
\[
    f'(x)\ge 0 \quad\text{for $-a/2\le x \le a/2$},
\]
then
\[
    \pdf{}x u(x,y) \ge 0 \quad \text{for $-a/2\le x \le a/2$}.
\]
\end{lemma}

\begin{proof}
Existence and uniqueness are standard. (The hypothesis that $0<b<\pi$ means we can use
a grim reaper surface as an upper barrier.)

For the ``furthermore" assertion, note by uniqueness that $u$ is invariant under
$(x,y)\mapsto (a/2-x,y)$ and under $(x,y)\mapsto (-a/2-x,y)$, so $\pdf{}xu(x,y)=0$ when $x=a/2$
and when $x=-a/2$.
By differentiating the translator equation~\eqref{translator-equation} with respect to $x$, we see that $v:=\pdf{}xu$ satisfies
an elliptic partial differential equation of the form
\[  
     a_{ij}(x,y) \D_{ij}v(x,y) + b_i(x,y) \D_iv(x,y) = 0.
\]
Hence by the maximum principle, 
\[
  v: [-a/2,a/2]\times [0,b] \to \RR
\]
attains its minimum on the boundary.
\end{proof}

\begin{proof}[Proof of Theorem~\ref{property-theorem}]
Let $f^i:\RR\to\RR$ be a sequence of smooth functions satisfying the hypotheses on $f$ in 
Lemma~\ref{property-lemma} such that
\begin{gather*}
f^1\le f^2 \le f^3\le\dots, \\
f^i(x)\equiv 0 \quad\text{for $-a/2<x<0$, and} \\
\lim_{i\to\infty}f^i(x)= \infty \quad\text{for $0<x<a/2$}.
\end{gather*}
Let $u^i$ be the $(2a,0)$-periodic solution of the translator equation corresponding to $f^i$ as in 
Lemma~\ref{property-lemma}.  By the maximum principle,
\[
    u^1 \le u^2 \le \dots \le u_{a,b}
\]
where $u_{a,b}$ is as in Proposition~\ref{first-existence-proposition}.
It follows that $u^i$ converges to $u_{a,b}$, and that the convergence is smooth
except at the points $(na,0)$, $n\in\ZZ$.   Thus by Lemma~\ref{property-lemma},
\[
  \pdf{}x u_{a,b} \ge 0 \quad\text{for $-a/2\le x \le a/2$}.
\]
Since $u_a$ was obtained as a limit of $u_{a,b_i}-z_i$ for $b_i\uparrow b(a)$
(and for suitable $z_i\in \RR$), we see that $\pdf{}xu_a\ge 0$ on $[-a/2,a/2]\times (0,b(a))$.
The strict inequalities in Theorem~\ref{property-theorem} follow by the strong maximum principle
and the strong boundary maximum principle.
\end{proof}

\begin{corollary}\label{mixed-curvature}
Every trident $M_a$ has points where the Gauss curvature (with respect to the Euclidean metric)
is negative and other points where it is positive.
\end{corollary}

By contrast, all graphical translators have nonnegative curvature everywhere
(by a theorem of Spruck and Xiao~\cite{spruck-xiao}),
and all known semigraphical translators other than tridents have negative curvature
 everywhere~\cite{scherkon}.

\begin{proof}
Let $u=u_a$ be as in Theorem~\ref{existence-theorem}.
By Corollary~\ref{invariance-corollary},  $u$ is invariant under 
\[
 (x,y)\mapsto (a/2-x,y),
\]
and thus $u_{xy}(a/2,y)\equiv 0$. 
By Theorem~\ref{property-theorem}, $u_{xx}(a/2,y)<0$ for all $y\in (0,b(a))$.

Since $u(a/2,0)=\infty$ and $u(a/2,b(a))= - \infty$, there exist $y$ for which $u_{yy}(a/2,y)<0$
and other $y$ for which $u_{yy}(a/2,y)>0$.
The Gauss curvature of $M_a$ is positive at the former and negative at the latter.
(Recall that the sign of the Gauss curvature is equal to the sign of $u_{xx}u_{yy}-u_{xy}^2$.)

(It is also easy to prove that the Gauss curvature is negative at all points in the vertical lines in $M_a$.)
\end{proof}

\section{Continuous Dependence on the Period}\label{dependence-section}

In this section (and in \S\ref{to-zero-section}), we think of $M_a$ 
(from Theorem~\ref{existence-theorem}) as a $(2a,0,0)$-periodic surface in $\RR^3$
rather than as a surface in a quotient of $\RR^3$.

\begin{theorem}\label{dependence-theorem}
Suppose that $a(i)$ converges to $a\in (0,\infty)$.
Then $M_{a(i)}$ converges smoothly to $M_a$, and $b(a(i))$ converges to $b(a)$.
\end{theorem}

\begin{proof}
By passing to a subsequence, we can assume 
that $b(a(i))$ converges to a limit $b$, and 
that $M_{a(i)}$ converges smoothly
to a limit $M$ that lies in the slab $\{-b \le y \le b\}$.

Exactly as in the proof of Theorem~\ref{existence-theorem}, one shows that $M$ has all the properties listed
in that theorem.   Thus by the Uniqueness Theorem~\ref{uniqueness-theorem}, $M=M_a$.

Next, we show that $b=b(a)$. 
Since $b(\cdot)$ is an increasing function (by Theorem~\ref{existence-theorem}),
  it suffices to show that $b(a(i))\to b(a)$
for sequences $a(i)\to a$ such that $a(i)/a$ is rational.  We assume that $a(i)$ is such a sequence.
 Note that
 \begin{align*}
 \Pi(M_{a(i)}) &= \RR\times (-b(a_i),b(a_i)), \, \text{and}\\
 \Pi(M_a) &= \RR\times (-b(a),b(a)),
\end{align*}
where $\Pi(x,y,z):=(x,y)$.
Thus, since $M_{a(i)} \to M_a$, we see that $b(a)\le b$.

It remains to show that $b\le b(a)$.  
Suppose to the contrary that $b>b(a)$.
Choose $\beta$ with $0<\beta<b(a)$ very close to $b(a)$, so that 
\begin{equation}\label{very-close}
    b(a) - \beta < b - b(a).
\end{equation}
The curvature estimates imply that $\Tan(M_a,p_k)\cdot \ee_2\to 0$
for every sequence of points $p_k=(x_k,y_k,z_k)$ in $M_a$ such that $y_k\to b(a)$.
Thus there is a $\beta'$ such that $\beta<\beta' < b(a)$ and such that
\[
    \max_x\pdf{u_a}y (x,\beta') < \min_x\pdf{u_a}y (x,\beta).
\]
Consequently, for all sufficiently large $i$,
\[
    \max_x \pdf{u_{a(i)}}y (x,\beta') < \min_x\pdf{u_a}y (x,\beta).
\]  
Fix such an $i$.  Let $\delta=\beta'-\beta$ and consider
\begin{align*}
&f: \RR\times [\beta, b(a)) \to \RR, \\
&f(x,y) = u_a(x,y) - u_{a(i)}(x, y + \delta).
\end{align*}
Since 
\[
\delta=\beta'-\beta<b(a)-\beta < b-b(a)
\]
by choice of $\beta$ (see~\eqref{very-close}),
we see that $y<b(a)$ implies that $y+\delta<b=\lim b(a_i)$, so that $f$ is defined
on the entire strip $\RR\times [\beta,b(a))$ if $i$ is sufficiently large.

Note that $f$ is periodic since $a(i)/a$ is rational, 
that $\pdf{f}y>0$ on the lower edge $\RR\times \{\beta\}$ of the domain,
and that $f$ is $-\infty$ on the upper edge $\RR\times \{b(a)\}$ of the domain.
Thus 
$f$ attains its maximum at an interior point, which is impossible
by the strong maximum principle.  
The contradiction proves that $b\le b(a)$.
\end{proof}

\section{Behavior as $a\to 0$}\label{to-zero-section}
\begin{figure}[htbp]
\begin{center}
\includegraphics[width=.5\textwidth]{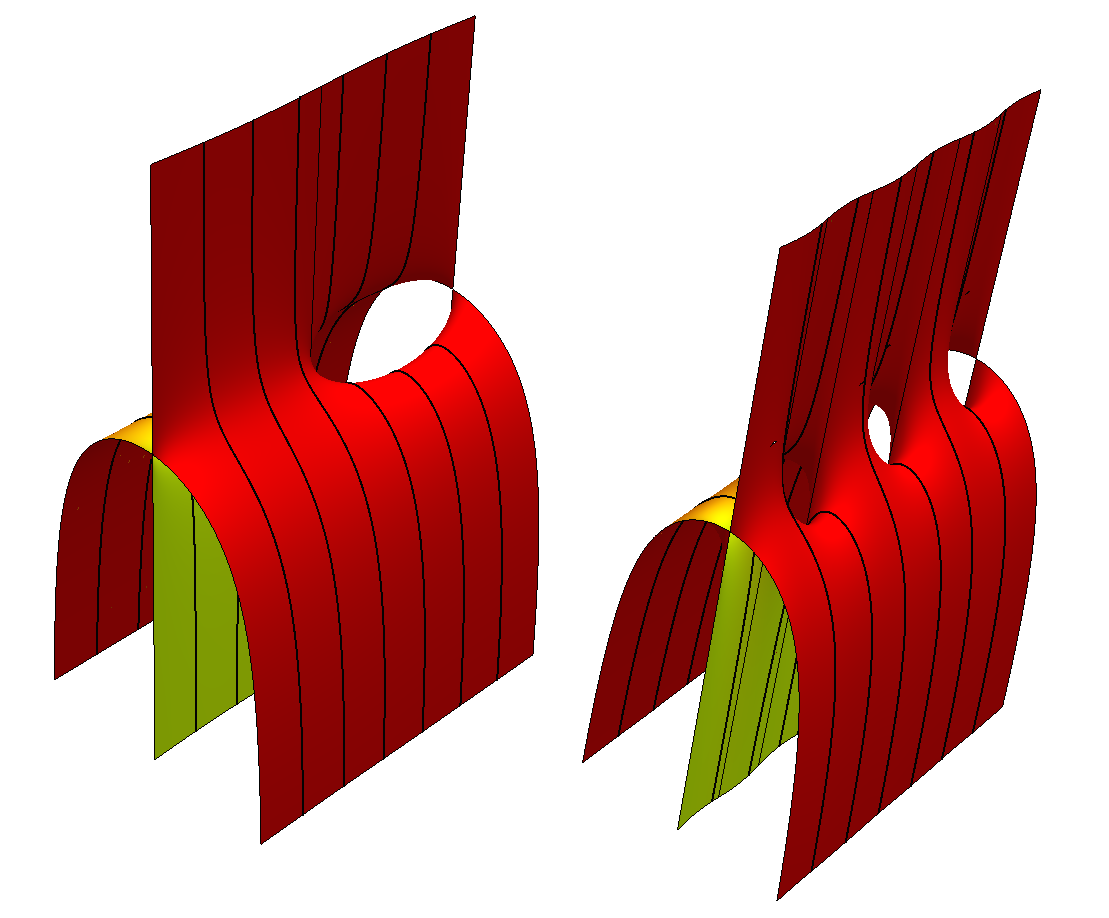}
\end{center}
\caption{The surfaces $M_3$ (left) and $M_1$ (right).} \label{fig:tres}
\end{figure}

In this section (as in~\S\ref{dependence-section}), 
 we think of $M_a$ as a $(2a,0,0)$-periodic surface in $\RR^3$
rather than as a surface in a quotient of $\RR^3$.

Recall that we distinguish $M_a$ from its vertical translates by requiring that
$\Tan(M_a,O)$ is the $yz$-plane.

\begin{theorem}\label{to-zero-theorem}
As $a\to 0$, the surface $M_a$ converges to the union of the $xz$-plane 
and the grim reaper surface 
\[
    \{(x,y,z): z = \log(\cos y), \, |y|<\pi/2\}.  
\]
The convergence is smooth away from the $x$-axis.
\end{theorem}

\begin{proof}
For small $a$, 
Nguyen's examples~\cite{nguyen-tridents} have the limiting behavior described
in Theorem~\ref{to-zero-theorem}. 
For such $a$, the Uniqueness Theorem~\ref{uniqueness-theorem} implies that the examples she constructs 
are the same as the examples in this paper.
\end{proof}

We now give a different proof of Theorem~\ref{to-zero-theorem} that is independent of Nguyen's construction.

\begin{lemma}\label{area-bound-lemma}
Let $M = M_a$. If $W$ is an open domain  with $\area_g(\partial W)<\infty$, then 
\begin{equation}\label{boundary-area}
\area_g(M \cap \overline{W}) < \area_g(W \cap \{y = 0\}) + \frac12 \area_g(\partial W).
\end{equation}
\end{lemma}

\begin{proof}
Note that 
\[
\area_g(M \cap \overline{W} \cap \{y \ge 0\}) < \frac12 \area_g(\partial(W \cap \{y > 0\}))
\]
since $M\cap\{y>0\}$ is a graph over a convex domain (namely a strip) and therefore is $g$-area-minimizing.
Likewise
\[
\area_g(M \cap \overline{W} \cap \{y \le 0\}) < \frac12 \area_g(\partial(W \cap \{y < 0\})).
\]
Adding these inequalities gives~\eqref{boundary-area}.
\end{proof}

\begin{lemma}\label{gauss-lemma}
Let 
\[
  L_a = \cup_{n\in\ZZ}\{(na, 0)\} \times \RR
\]
be the union of the vertical lines in $M_a$.
Let $L^+_a=L_a\cap \{z>0\}$ and $L_a^-=L_a\cap \{z<0\}$.
The image of $L^+_a$ under the Gauss map of $M_a$
 is contained in $Q\cup Q'$, where $Q$ and $Q'$ are quarter-circles in the
  equator $E = \Ss^2\cap\{ z = 0\}$. 
  The same holds (with different quarter-circles) for $L^-_a$.
\end{lemma}

\begin{proof}
Since $M\cap\{y>0\}$ is a graph over a strip,
 $\Tan(M,(0,0,z))$ turns through an angle of $\pi/2$ as $z$ goes from $0$ to $\infty$.
  Thus the image of $\{(0, 0)\}\times (0, \infty)$ under the Gauss Map of $M$ is a quarter-circle $Q$ in $E$.
  Likewise the Gaussian image of $\{(a,0)\}\times (0,\infty)$ is a quarter-circle $Q'$ in $E$.
   By the periodicity, the Gauss map image of $\{(na, 0)\}\times (0, \infty)$ is $Q$ if $n$ is even 
   and is $Q'$ if $n$ is odd. 
\end{proof}

\begin{proposition}\label{curvature-bound}
There is a constant $C < \infty$ (independent of $a$) such that
\[
|A(M_a, p)| \, \min\{1, \dist(p, X)\} \le C.
\]
Here $|A(M,p)|$ is the norm of the second fundamental form with respect to the
Euclidean metric, and $\dist(p,X)$ is the Euclidean distance from p to the $x$-axis $X$.
\end{proposition}

\begin{proof}
Suppose the proposition is false. Then there is a sequence of examples $p_i \in M_{a(i)}$ 
such that
\begin{equation}\label{blowup}
|A(M_{a(i)}, p_i)|\, \min\{1, \dist(p_i, X)\} \to \infty.
\end{equation}
 We may choose $p_i$ to maximize the left-hand side of~\eqref{blowup} since
\[
\text{$|A(M_a,(x,y,z))|  \to  0$ as $z \to \pm \infty$.}
\]
Let $\lambda_i:=|A(M_{a(i)},p_i)|$, and let
\begin{align*}
&\phi_i: \RR^3 \to \RR^3, \\
&\phi_i(p) = \lambda_i(p-p_i).
\end{align*}
Let $M_i'=\phi_i(M_{a(i)})$ and $X_i'=\phi_i(X)$.
Thus 
 \[
   |A(M_i',O)|=1,
\]
where $O$ is the origin in $\RR^3$.
Note that the left side of~\eqref{blowup} is
\[
   \lambda_i \,\min\{ 1, \dist(p_i,X)\} = \min\{\lambda_i, \lambda_i\dist(p_i, X)\} = \min\{ \lambda_i, \dist(O,X_i')\},
\]
so (by~\eqref{blowup})
 \begin{equation}\label{X-goes-away}
   \dist(O,X_i') \to \infty.
\end{equation}
Note that~\eqref{X-goes-away} implies (perhaps after passing a  subsequence) that
\begin{equation}\label{either-or}
\begin{gathered}
\text{$\dist(O, \phi_i(L^+_{a(i)})) \to \infty$, or} \\
\dist(O, \phi_i(L^-_{a(i)})) \to \infty,
\end{gathered}
\end{equation}
where $L^\pm_{a(i)}$ are as in Lemma~\ref{gauss-lemma}.
 After passing to a further subsequence, the $M_i'$ converge smoothly to a complete (Euclidean) 
 minimal surface $M'$ with 
 \begin{equation}\label{not-flat}
 |A(M',O)| = 1.
 \end{equation}
  Let $\nu$ be the Gauss map of $M'$.

By~\eqref{either-or} and Lemma~\ref{gauss-lemma}, 
\[
\nu(L) \subset Q \cup Q'
\]
 for certain quarter circles $Q$ and $Q'$ in $E$, 
  where $L$ 
  is the union of the vertical lines in $M'$. Since each component of $M'\setminus L$ is a graph,
  \[
     \nu(M')\cap E = \nu(L)\cap E.
  \]
 Thus
\[
  \nu(M') \cap E \subset Q \cup Q'.
\]
But this is impossible since a complete minimal surface whose Gauss map misses
an arc must be a plane~\cite{osserman-book}[Theorem~8.2], contradicting~\eqref{not-flat}.
(Indeed, if the Gauss map misses even  $5$ points,
then the surface is a plane~\cite{fujimoto}.)

(In our setting, one can use more elementary results than~\cite[Theorem~8.2]{osserman-book} or
\cite{fujimoto}.  Note that the $M'$ above has quadratic area growth by monotonicity.
Using that, it is not hard to prove (as in~\cite{qing-chen} or~\cite{li-quadratic}) 
that the quotient $\widetilde M'$ of $M'$ by the period has finite total curvature,
from which it follows that $\widetilde M'$ is conformally a punctured Riemann surface, that the Gauss
map extends smoothly to the punctures 
(see for instance~\cite[\S9]{osserman-book} or~\cite[Theorem 17]{white-parkcity}),
 and hence that the Gauss map can miss at most finitely many
points in $\Ss^2$.)
\end{proof}

\begin{proof}[Proof of Theorem~\ref{to-zero-theorem}]
Recall (by Theorem~\ref{existence-theorem}) that $b(a)\le \pi$ for each $a>0$ and that $b(\cdot)$ 
is an increasing function.  Thus
\begin{equation}\label{less-than-pi}
 b:=\lim_{a\to 0}b(a) \le \pi.
\end{equation}

Let $a(i)$ be a sequence converging to $0$.
Fix a small $\eps>0$ and  let $(0,\eps,z_i)$ be the unique point of intersection of $M_{a(i)}$
and the line $\{(0,\eps)\}\times \RR$.
By passing to a subsequence, we can assume that $z_i$ converges to a limit $\hat{z}\in [-\infty,\infty]$
and that the surfaces $M_{a(i)}-(0,0,z_i)$ converge smoothly away from the line
$X-(0,0,\hat{z})$ to a limit translator $M'$.  (If $\hat{z}=\pm \infty$, then the convergence
is smooth everywhere.)
 Since $a(i)\to 0$, $M'$ is translation invariant in the $x$-direction.
 
 Let $\Sigma$ be the component of $M'\cap\{y>0\}$ containing $(0,\eps,0)$.
 Then $\Sigma$ is either the plane $\{y=\eps\}$ or $G\cap\{y>0\}$ for some grim reaper surface $G$.
Since $M'$ contains the line $Z$, $M'$ also contains the entire $xz$-plane $\{y=0\}$.

We claim that $\Sigma$ is not the plane $\{y=\eps\}$ (provided $\eps$ is sufficiently small).
For if it were, then by symmetry $M'$ would contain the plane $\{y=-\eps\}$.
Let $\Rr=(0,1)\times (-\eps,\eps)$ and $\Omega=\Rr\times (0,\infty)$.
By Lemma~\ref{area-bound-lemma},
\begin{equation}\label{area-Omega}
  \area_g(M'\cap \overline{\Omega}) 
  \le \area_g(\Omega\cap\{y=0\}) + \frac12 \area_g(\partial \Omega)
\end{equation}
By direct calculation (or by Lemma~\ref{prelim-lemma}),
\begin{equation}\label{reminder}
   \area_g(\{y=c\}\cap\overline{\Omega}) = 1
\end{equation}
for $|c|\le\eps$.
Since $M'$ contains the planes $y=0$, $y=\eps$, and $y= -\eps$, \eqref{area-Omega} and~\eqref{reminder}
imply 
\begin{align*}
3 
&\le 1 + \frac12\area_g(\partial \Omega)
\\
&= 1 + \frac12( \area_g(\Rr\times\{0\}) + \area_g((\partial \Rr)\times [0,\infty))  )
\\
&= 1 + \frac12(\area(\Rr) + \length(\partial \Rr))
\\
&= 1 + \frac12(2\eps + 2 + 4\eps)
\\
&= 2 + 3\eps
\end{align*}
so $\eps\ge \frac13$.

Thus by choosing $\eps<\frac13$, we guarantee that $\Sigma$ is not a vertical plane.
Consequently $\Sigma$ has the form
\begin{equation}\label{the-form}
  \Sigma = \graph(f)\cap \{y>0\}
\end{equation}
where
\[
   f(x,y)= \log(\cos(y-y_0)) + c
\]
for some $y_0>0$ and $c$.
Since $\Sigma\subset \{0<y<b\}$ and $b\le \pi$, we see that 
\[
   y_0\le  b - \pi/2 \le \pi/2.
\]

We claim that $y_0< \pi/2$.   For if $y_0=\pi/2$, then $\Sigma$ would be a grim
reaper surface defined over the strip $0<y<\pi$.  Thus $M'$ would be the union
of the $xz$-plane, the grim reaper surface $\Sigma$, and the image of $\Sigma$ under
reflection in the $xz$-plane. 
 The limit $M''$ of $M'+(0,0,\lambda)$ as $\lambda\to\infty$
would then consist of the planes $y=\pm \pi$ with multiplicity $1$ and the plane $y=0$
with multiplicity $3$.  However, this is impossible.
We showed above that $M'$ cannot contain the planes $y=0$, $y=\eps$, and $y= -\eps$
if $\eps$ is small.
Exactly the same proof shows that $M''$ cannot contain the plane $y=0$ with multiplicity~$3$.
This completes the proof that
$y_0< \pi/2$.

Since $y_0<\pi/2$, we see that 
$\graph(f)$ (in~\eqref{the-form}) intersects the plane $y=0$ in a horizontal straight line $L$.
Since $\graph(f)\cap\{y>0\}$ and the $xz$-plane are both contained in $M'$,
we see that the convergence of $M_{a(i)} - (0,0,z_i)$ to $M'$ is not smooth
along $L$.  Thus $L=X-(0,0,\hat{z})$.  In particular, $|\hat{z}|<\infty$.

Thus $M_{a(i)}$ converges to $M:=M'+(0,0,\hat{z})$.
Note that $M$ consists of the $xz$-plane together with $G$ and $G'$
where $G$ is the $\{y\ge 0\}$ portion of a grim reaper surface that contains $X$
and where $G'$ is the image of $G$ under reflection in the $xz$-plane.

We claim that $G$ is exactly half of a grim reaper surface, so that
$G\cup G'$ is a grim reaper surface.
For readers familiar with varifolds, this is an immediate consequence of the
fact that $M$ (a limit of stationary varifolds) is a stationary varifold.

Here is a different proof.  Suppose $G$ were not half of a grim reaper.
Then the tangent cone $C$ to $M$ at the origin would consist of the $xz$-plane
together with 
\[
   \{ (x,y,z):   z = |y|/c \}
\]
for some constant $c\ne 0$.  Assume $c<0$. (The case $c>0$ is handled in the same way.)
Let $L$ be large and let
\[
  \Omega=\Omega_L = \{(x,y,z): \text{$0\le x\le L$ and $-1\le z \le |y|/c$} \}.
\]
Note that
\[
  \area(\partial \Omega) = 2(1+c^2)^{1/2}L + 2cL + c.
\]
Thus by Lemma~\ref{area-bound-lemma},
\begin{align*}
\area(M\cap\overline{\Omega}) 
&\le
\area(\Omega\cap\{y=0\}) + \frac12(\area(\partial \Omega)
\end{align*}
i.e.,
\[
2(1+c^2)^{1/2}L + L 
\le
L + \frac12(2(1+c^2)^{1/2}L + 2cL + c)
\]
or
\[
(1+c^2)^{1/2}L \le cL + \frac12c.
\]
Dividing by $L$ and then letting $L\to\infty$ gives $(1+c^2)^{1/2}\le c$, a contradiction.
\end{proof}

\begin{theorem}\label{pi-over-two-theorem}
$\lim_{a\to 0}b(a)=\pi/2$.
\end{theorem}

The proof is exactly like the proof of Theorem~\ref{dependence-theorem}, so we omit it.

\section{Behavior as $a\to\infty$}\label{to-infinity-section}

\begin{theorem}\label{to-infinity-theorem}
Every sequence of real numbers tending to infinity has a subsequence $a(i)$
for which $M_{a(i)}$ converges smoothly to a limit translator $M$.
If $M$ is a such a limit, then $M$ contains
$Z$, 
$M$ is tangent to the $xz$-plane at the origin,  and $M\cap\{y>0\}$ 
is the graph of a function $u:\RR\times(0,\pi)\to \RR$ with boundary values
\begin{align*}
u(x,0) &= -\infty \quad\text{if $x<0$}, \\
u(x,0) &= \infty \quad\text{if $x>0$}, \\
u(x,\pi) &= -\infty \quad\text{for all $x$}.
\end{align*}
As $\hat{x}\to -\infty$,
\[
     u(\hat{x}+x,y)-u(\hat{x},y)
\]
converges smoothly to $\log(\sin y)$.
If $y_i\in (0,\pi)$ and $x_i\to \infty$, then 
\[
  (M\cap\{y\ge 0\}) - (x_i,y_i,u(x_i,y_i))
\]
converges smoothly to plane $x=0$.
Furthermore, $M$ is negatively curved everywhere, and the Gauss map
induces a diffeomorphism from $M\cap\{y\ge 0\}$ onto
\[
    (\Ss^2\cap\{z\ge 0\}) \setminus \{\ee_2\}.
\]
\end{theorem}

\begin{proof}
Exactly as in the proof of the Main Existence Theorem~\ref{existence-theorem},
there is a subsequence $a(i)$ for which $M_{a(i)}$ converges smoothly  to a smooth surface $M$
such that $M$ contains $Z$, such that 
 $M$ is tangent to the $yz$-plane at the origin, and such that $M\cap\{z>0\}$
is the graph of a function $u$ defined on a strip $S:=\RR\times (0,\beta)$ where
\[
  \beta \le \lim b(a(i)) \le \pi.
\]
On each component of $(\partial S)\setminus\{(0,0)\}$, $u$ is either $+\infty$ or $-\infty$.
Note that $\Tan(M,(0,0,z))$ rotates clockwise as $z$ increases (because each $M_{a(i)}$ has
that property), so
\[
u(x,0) 
= 
\begin{cases}
-\infty &\text{for $x<0$, and} \\
+\infty &\text{for $x>0$}.
\end{cases}
\]
By Theorem~8.1 in~\cite{scherkon}, $u(x,\beta)=-\infty$. 
The remaining assertions are proved in Theorem~12.1 of~\cite{scherkon}.
\end{proof}

\begin{corollary}\label{to-infinity-corollary}
$\lim_{a\to\infty}b(a)=\pi$.
\end{corollary}

The corollary follows immediately from Theorem~\ref{to-infinity-theorem} and the fact that $b(a)$ is an
increasing function of $a$ (by Theorem~\ref{existence-theorem}).

\begin{conjecture}\label{pitchfork-conjecture}
There is a unique surface $M$ satisfying the conclusions of Theorem~\ref{to-infinity-theorem}.
\end{conjecture}

In~\cite{scherkon} (see Theorem~2.3), a surface satisfying $M$ the conclusions of Theorem~\ref{to-infinity-theorem}
is called a {\bf pitchfork of width $\pi$}.  See Figure~\ref{semigraphical-figure}.

If Conjecture~\ref{pitchfork-conjecture} is true, then it is not necessary to pass to a subsequence in Theorem~\ref{to-infinity-theorem}.
Also, in~\cite{scherkon}, such a pitchfork was obtained as a limit of so-called Scherkenoids.
If the uniqueness conjecture is true, then that limit of Scherkenoids is the same
as the surface in Theorem~\ref{to-infinity-theorem} (obtained as a limit of tridents.)

\section{Widths}\label{widths-section}

\begin{theorem} The set $B:=\{b(a): a\in (0,\infty)\}$ is the interval $(\pi/2,\pi)$.
\end{theorem}

\begin{proof}
Recall that $b(\cdot)$ is a continuous, increasing function
 (by Theorems~\ref{existence-theorem} and~\ref{dependence-theorem}),
 that $\lim_{a\to 0}b(a)=\pi/2$ (by Theorem~\ref{pi-over-two-theorem}),
and that $\lim_{a\to\infty}b(a)=\pi$ (by Corollary~\ref{to-infinity-corollary}).  
Hence
\[
   (\pi/2,\pi) \subset B \subset [\pi/2, \pi].
\]
By Theorems~\ref{sharp} and~\ref{upper-sharp} below, $b(a)$ cannot
take the values $\pi/2$ or $\pi$. 
\end{proof}

\begin{theorem}\label{sharp} For every $a$, $b(a)>\pi/2$.
\end{theorem}

\begin{proof}
By definition of $b(a)$ (see Theorem~\ref{existence-theorem}),
 the assertion is equivalent to the assertion that $\alpha(a, \pi/2)<3a$.
For $0<b<\pi/2$, 
consider the grim reaper surface 
\begin{equation}\label{reaper}
z=\log(\cos y)+c_b,
\quad
\text{where $c_b= - \log(\cos b)$.}
\end{equation}
Note that $c_b>0$.
The surface intersects the plane $y=0$ in the line $z=c_b$ and it intersects
the plane $z=0$ in the lines $y=\pm b$.
Let $G(b)$ be the portion of the surface in the region 
\[
 \{ -a\le x \le a\} \cap \{y\ge 0\} \cap \{z\ge 0\}.
\]
We can regard $G(b)$ as a surface in $K_{a,b}\times\RR$ (or in $K_{a,b'}\times\RR$ if $b'\ge b$.)
By Lemma~\ref{prelim-lemma},
\begin{equation}\label{less-than-2a}
  \area(G(b)) <  2a.
\end{equation}
(If~\eqref{less-than-2a} is not clear,
 note that $\nu\cdot \vv< 1$ on the portion of the boundary where $y=b$ and that $\nu\cdot\vv=0$
on the rest of the boundary.)

If $I$ and $J$ are intervals, we let $I\otimes J$ denote 
the rectangle
$
    I \times \{0\} \times J
$
in the plane $y=0$.

We claim that
\begin{equation}\label{first-comparison}
\alpha(a,b) < \area(G(b)) + \area([0,a] \otimes [0,\infty))
\end{equation}
for $0<b<\pi/2$.
(We remark that~\eqref{first-comparison} together with~\eqref{less-than-2a}
already establishes that $\alpha(a,b)<3a$ for $b<\pi/2$.)

We prove~\eqref{first-comparison} as follows.
Let $\Sigma(b)$ be the union
of the rectangle
\[
     [-a,0] \otimes [0,c_b]
\]
and the strip
\[
    [0,a]  \otimes  [c_b,\infty)
\]
(where $c_b$ is as in~\eqref{reaper}).
Clearly
\[
  \area(\Sigma(b)) = \area([0,a] \otimes  [0,\infty)).
\]
Note that $G(b)\cup \Sigma(b)$ has
boundary $\Gamma_{a,b}$.  Since $G(b)\cup \Sigma(b)$ is not smooth (but only
piecewise smooth), it is not an area-minimizing surface.  Thus
\begin{align*}
\alpha(a,b)
&<
\area(G(b)\cup \Sigma(b))  \\
&=
 \area(G(b)) + \area([0,a]  \otimes [0,\infty)).
\end{align*}
Thus we have proved~\eqref{first-comparison}.

As a special case of~\eqref{first-comparison}, we have
\begin{equation*}
\alpha(a,1) < \area(G(1)) + \area([0,a] \otimes [0,\infty)).
\end{equation*}
Thus by continuity of the function $\alpha$ (Proposition~\ref{alpha-bounds}),
\begin{equation}\label{improved-comparison}
\alpha(a,b) < \area(G(1)) + \area([0,a] \otimes [0,\infty)).
\end{equation}
for all $b$ sufficiently close to $1$.  

Now fix a $b$ slightly larger than $1$ for which~\eqref{improved-comparison} holds.
Let $S$ be a least-area surface with boundary $\Gamma(a,b)$.
Then
\begin{equation}\label{tophat}
 \area(S) =\alpha(a,b) < \area(G(1)) + \area([0,a] \otimes [0,\infty)).
\end{equation}
Let $\eps = b - 1$ and let $b' = \pi/2 - \eps$.

Let $\tau$ be the height at which 
\[
  \text{$G(b')\cap \{z=\tau\}$ is the line $\{y=b\}\cap \{z=\tau\}$.}
\]

Shifting the surfaces in~\eqref{tophat} by $\tau \,\ee_3$ gives
\begin{align*}
\area(S+\tau \, \ee_3) 
&< \area(G(1)+ \tau \, \ee_3) + \area( [0,a] \otimes [\tau,\infty)) \\
&= \area(G(b')\cap \{z\ge \tau \}) + \area([0,a] \otimes [\tau,\infty)).
\end{align*}

Now consider the surface 
\[
  M:= (S + \tau \, \ee_3) \cup ([-a,0] \otimes [0,\tau]) \cup (G(b')\cap\{z \le \tau \} + \eps \,  \ee_2).
\]   
Note that $M$  is a piecewise smooth surface with boundary $\Gamma(a,\pi/2)$.

Thus
\begin{align*}
\alpha(a,\pi/2)
&\le
\area(M) \\
&=
\area(S+ \tau \, \ee_3) 
\\ 
&\qquad + \area([-a,0] \otimes  [0,\tau]) 
 + \area(G(b')\cap\{z \le \tau \})
\\
&<
\area(G(b')\cap\{z\ge \tau\})  + \area([0,a]\otimes[\tau,\infty)) 
\\
&\qquad + \area([-a,0]\otimes[0,\tau]) + \area(G(b')\cap\{ z \le \tau\})  \\
&= \area(G(b')) + \area([0,a] \otimes  [0,\infty))  \\
&< 3a
\end{align*}

by~\eqref{less-than-2a}.
\end{proof}

\begin{theorem}\label{upper-sharp}
For every $a$, $b(a)<\pi$.
\end{theorem}

\begin{proof}
Fix an $a>0$.  Let $u=u_a$, $b=b(a)$, and
\begin{align*}
&w: (0,b)\to\RR, \\
&w(y) = u(-a/2,y).
\end{align*}
Consider the translator equation
\begin{equation*}
  \Delta u -  \frac{\D_iu\,\D_ju}{1+|\D u|^2}\D_{ij}u = -1.
\end{equation*}
or, equivalently,
\begin{equation}\label{translator-equation}
 (1+u_y^2)u_{xx} + (1+u_x^2)u_{yy} - 2u_xu_yu_{xy} = -(1 + |\D u|^2).
\end{equation}
On the line $x=(-a/2)$, we have $u_x=0$, $u_{xy}=0$, and $u_{xx}>0$
 (by Theorem~\ref{property-theorem}), so
\[
   u_{yy}  < - (1 + u_y^2) \quad \text{for $x=(-a/2)$.}
\]
Thus
\[
    w'' + (w')^2 + 1 < 0.
\]
In other words, $w$ is a strict supersolution of the translator equation.
Theorem~\ref{upper-sharp} now follows from Lemma~\ref{super} below.
\end{proof}

\begin{lemma}\label{super}
Suppose that $w: (c,d)\to\RR$ is a strict supersolution of the translator equation
and that $\lim_{x\to c}w(x)=\lim_{x\to d}w(x)= -\infty$.
Then $c-d< \pi$.
\end{lemma}

\begin{proof}
By translating, we may suppose that $w$ attains its maximum at $0$ and that $w(0)=0$.
Since $w$ is concave downward, we see that $w| (c,0]$ is strictly increasing.
Let $W$ be the graph of $w|(c,0]$:
\[
   W:= \{ (x, w(x)): c< x\le 0\}.
\]
Let $\Gamma$ be the left half of the standard grim reaper curve:
\[
  \Gamma:= \{ (x, \log(\cos x)): -\pi/2<x\le 0\}.
\]

The assertion that $w$ is a a strict supersolution of the translator equation
is equivalent to the assertion that the geodesic curvature of $W$ with respect
the translator metric is a positive multiple of the downward pointing unit normal.
   Of course $\Gamma$ is a geodesic for
the translator metric.

Thus near $(0,0)$, the curve $W$ lies between $\Gamma$ and the line $x=0$.

Suppose $c\le -\pi/2$.  Then there would be a maximum $s>0$ such that $\Gamma+(s,0)$
intersects $W$, and the point of contact would violate the maximum principle, a contradiction.

Thus $-\pi/2< c$.  
Likewise, $d< \pi/2$.
\end{proof}

\section{Semigraphical Translators}\label{semigraphical-section}

A translator is $M$ is called {\bf semigraphical} if 
\begin{enumerate}[\upshape (1)]
\item $M$ is a smooth, connected, properly embedded submanifold (without boundary) in $\RR^3$.
\item $M$ contains a nonempty, discrete collection of vertical lines.
\item $M\setminus L$ is a graph, where $L$ is the union of the vertical lines in $M$.
\end{enumerate}

Suppose $M$ is a semigraphical translator.  We may suppose without loss of generality 
that $M$ contains the $z$-axis $Z$.
Note that $M$ is invariant under $180^\circ$ rotation about each line in $L$,
from which it follows that $L\cap\{z=0\}$ is an additive subgroup of $\RR^2$.
The curvature estimates imply that $M-(0,0,\lambda)$ converges smoothly (perhaps
after passing to a subsequence) to an embedded translator $M_\infty$. 
Note that the limit translator cannot have any point where the tangent plane is non-vertical.
Thus $M_\infty$ is a union of one or more parallel vertical planes.
Likewise $M_{-\infty}$ (the limit as $\lambda\to -\infty)$ is the union of one or more parallel 
vertical planes.
Hence if $\Sigma$ is a connected component of $M\setminus L$, then $\Sigma$
is the graph of a function $u:\Omega\to \RR$, where $\Omega$ is one
of the components of 
\[
      \RR^2 \setminus \Pi(M_\infty\cup M_{-\infty}).
\]
Here $\Pi$ is the projection $\Pi(x,y,z)=(x,y)$.
Note that such an $\Omega$ (i.e., a component of $\RR^2$ minus two families of parallel lines)
must be one of the following (after a rigid motion of $\RR^2$):
\begin{enumerate}
\item A parallelogram.  Such translators are called ``Scherk translators" and were
completely classified in~\cite{scherkon}.
\item A semi-infinite parallelogram, i.e., a set of the form $\{(x,y): 0<y<w, \, x>my\}$ for some $m\ne 0$.
 Such translators are called ``Scherkenoids" and were completely classified in~\cite{scherkon}.
 In particular, for each $m$, there exists such a surface if and only if $w\ge \pi$, 
 and it is unique up to vertical translation.
\item An infinite strip $\RR\times (0,b)$ for some $b<\infty$.  There are three subcases, which we discuss below.
\item A wedge, i.e., a set of the form $\{(r\cos\theta,r\sin\theta): r>0, \, 0< \theta< \alpha\}$
for some $\alpha$ with $0<\alpha< \pi$.
      This case cannot occur; see Lemma~\ref{wedge-lemma} below.    
\item A halfplane.  In this case, $M$ contains only one vertical line by Remark~\ref{periodic-remark}
and Lemma~\ref{one-line-lemma} below.
     We conjecture that this case cannot occur.
\end{enumerate}
\begin{figure}[h]\label{semigraphical-figure}
{\includegraphics[width=3.8cm]{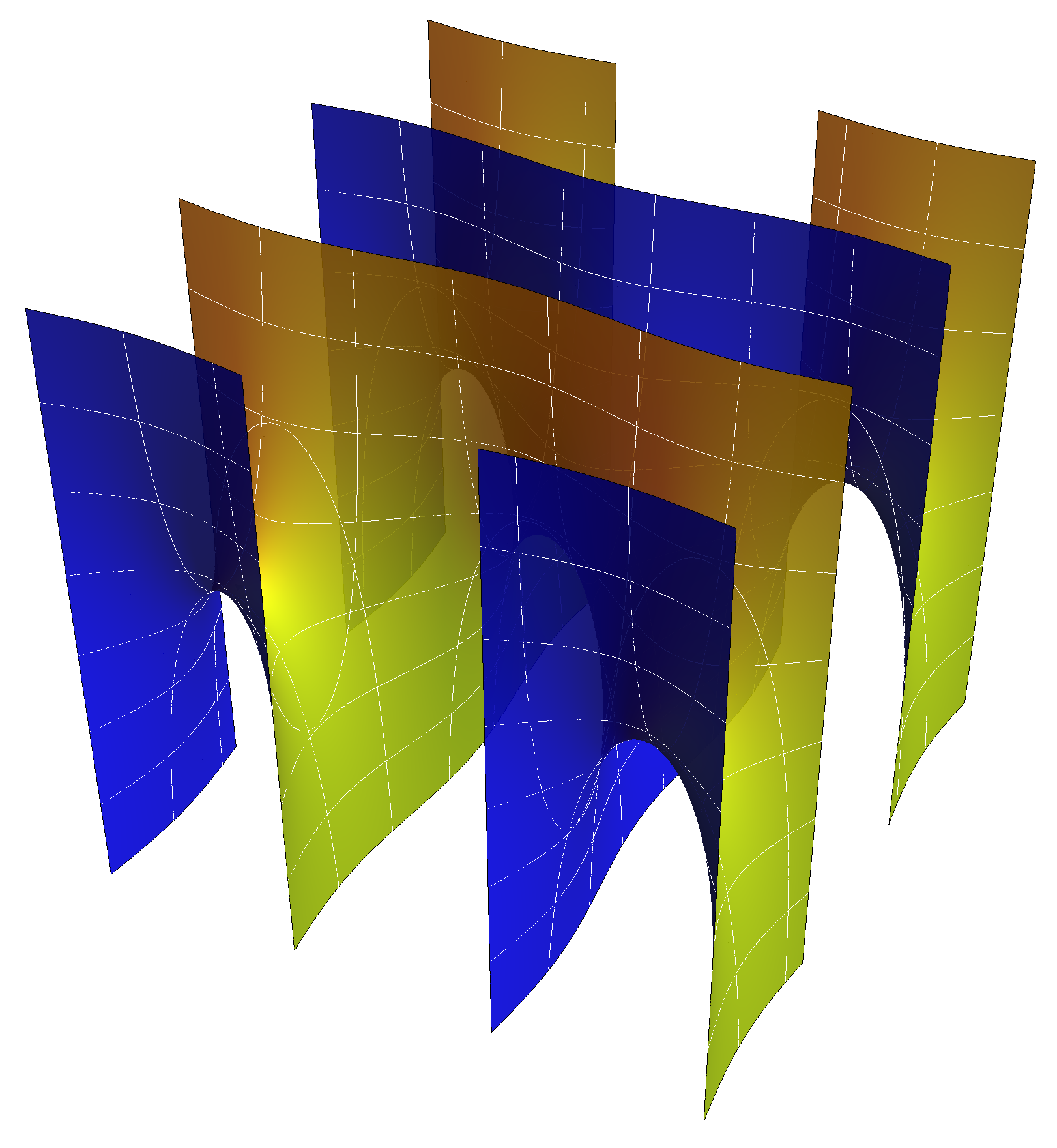}}
{\includegraphics[width=4cm]{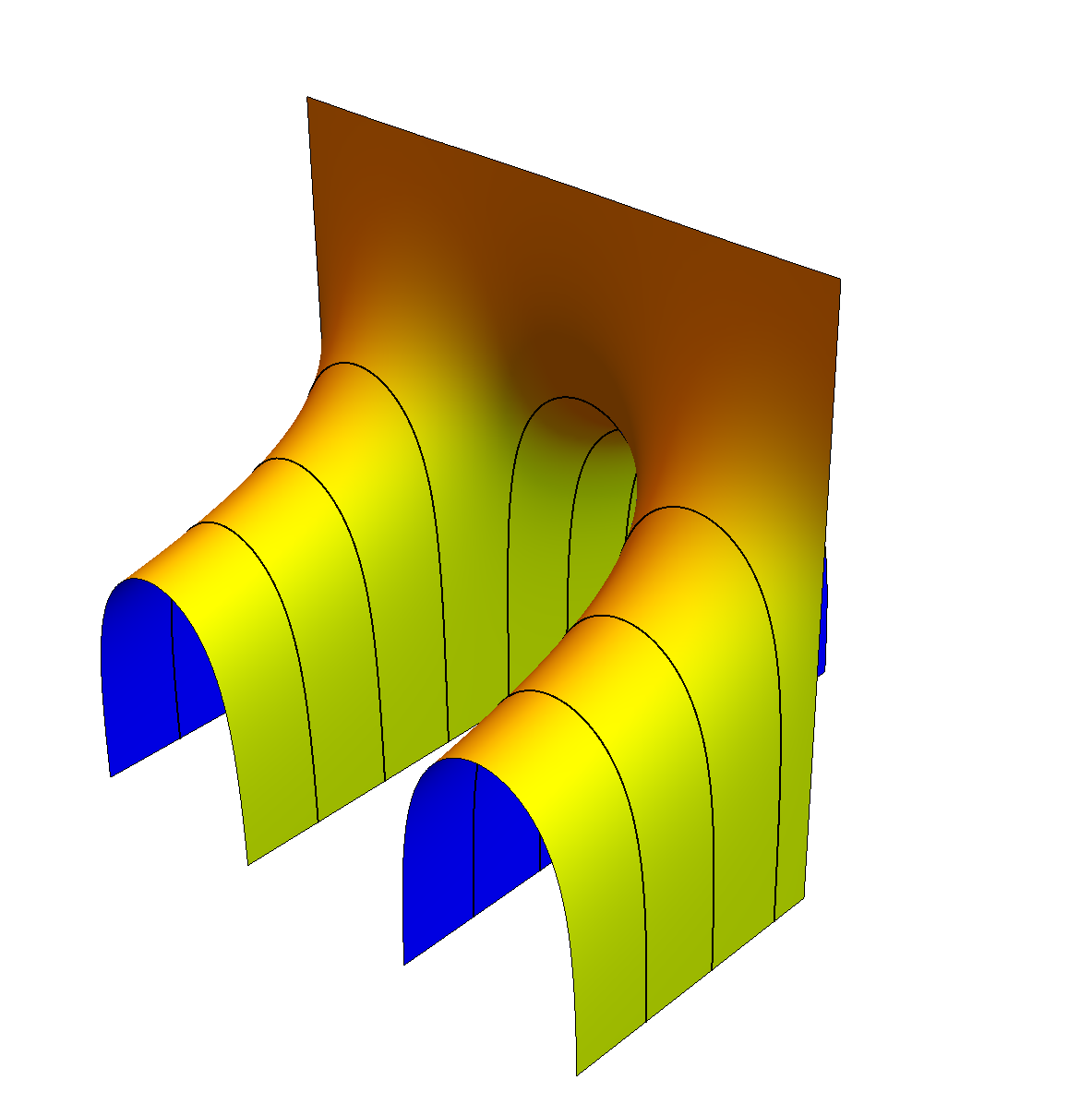}}
{\includegraphics[width=4cm]{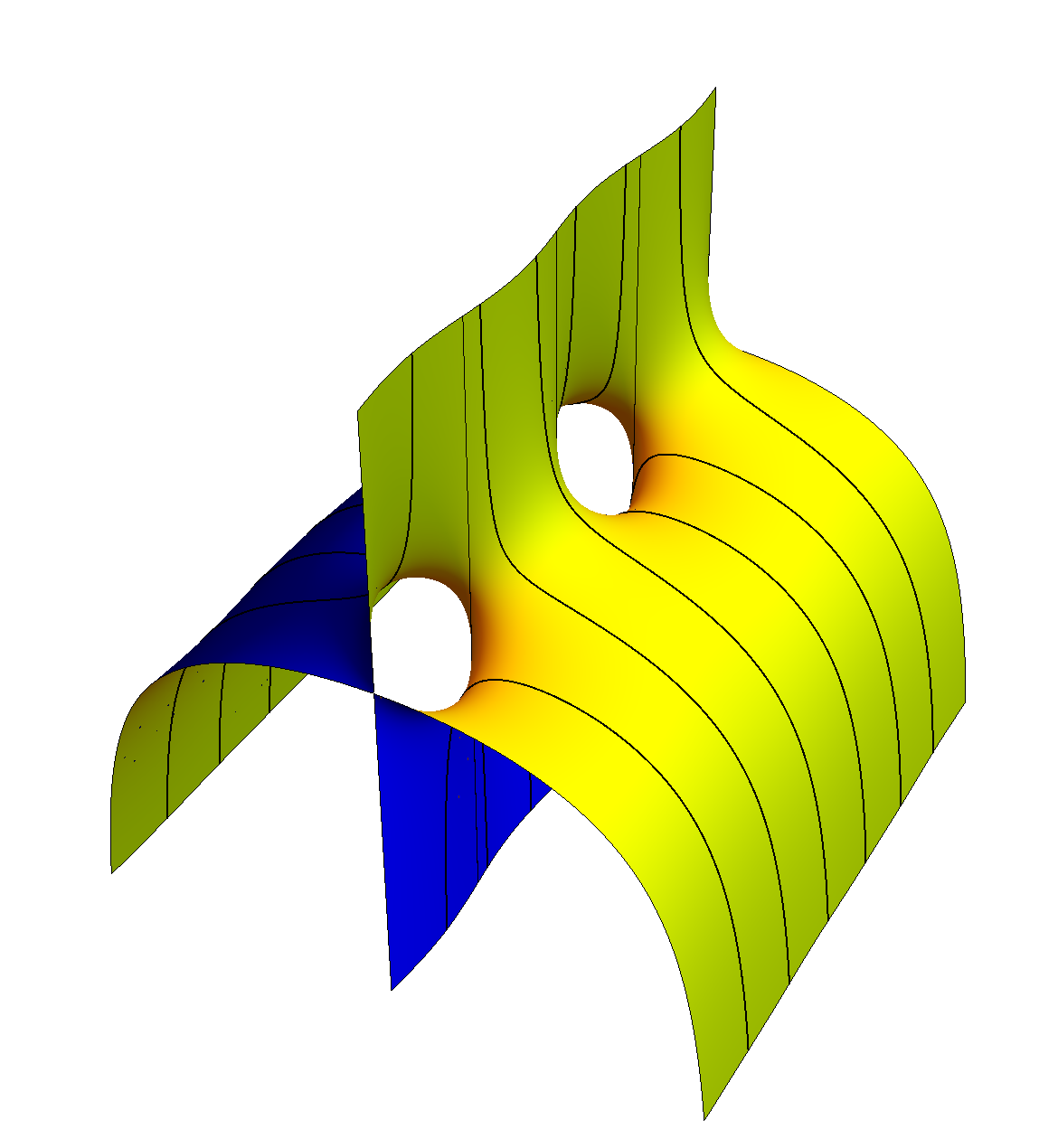}}
{\includegraphics[width=3.8cm]{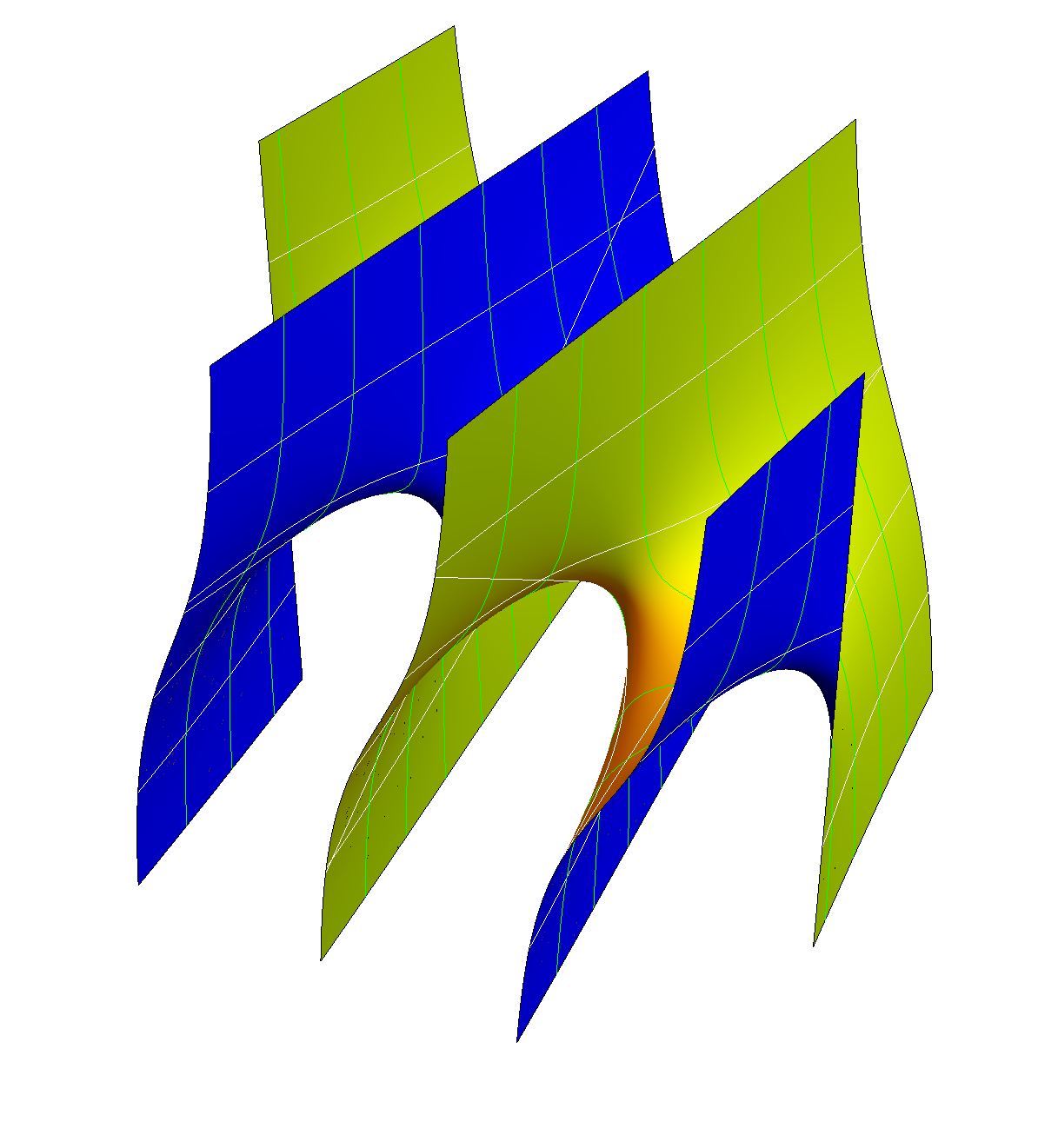}}
{\includegraphics[width=4cm]{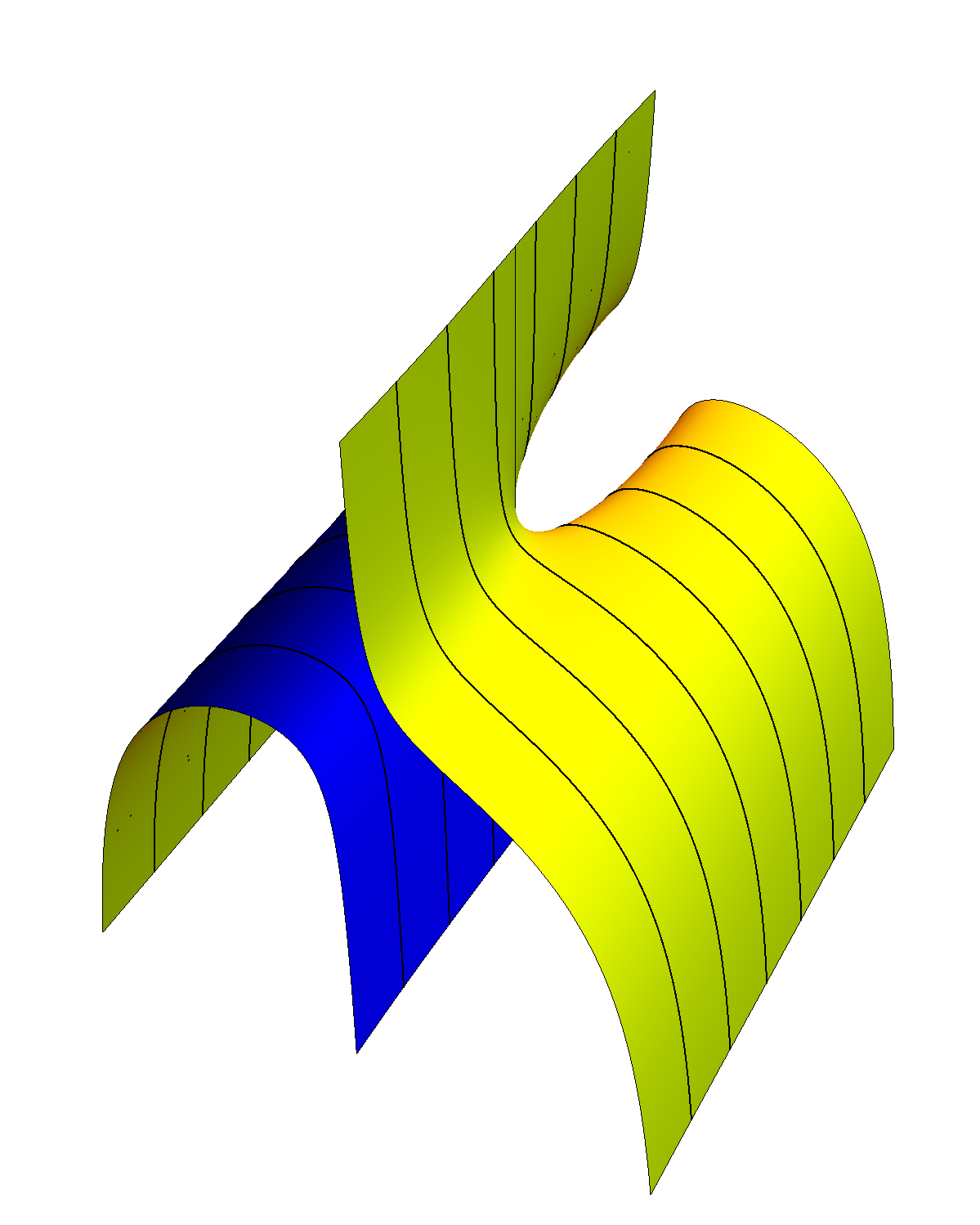}}
\caption{Semigraphical Translators.
Row 1 (left to right): Scherk translator, Scherkenoid, and Trident.
Row 2: Helicoid-like Translator and Pitchfork.}
\label{sequence2}
\end{figure}

\begin{remark}\label{periodic-remark}
Suppose $\Omega=\RR\times (0,b)$ for some $0<b\le\infty$ (so $\Omega$  is a strip or a halfplane.)
Let $S$ be the set of points $p$ in $\partial\Omega$ such that $M$ contains the vertical line $\{p\}\times \RR$.
If the $x$-axis contains a second point $(a,0)$ in $S$ (in addition to the origin),
then $M$ is periodic with period $(2a,0,0)$ and thus the $x$-axis contains infinitely many points of $S$.
\end{remark}

Now we discuss the case of a strip, i.e., the case when $\Omega=\RR\times (0,b)$ for some $0<b<\infty$.
Let $S$ be as in Remark~\ref{periodic-remark}.
There are three subcases, according to whether $S$ has exactly $1$ point, exactly $2$ points, or more
than $2$ points.

If $\Omega=\RR\times(0,b)$ and $S$ has exactly one point (namely the origin), $M$ is called a {\bf pitchfork}.
In this case, $u(\cdot,b)= -\infty$, and $u$ is $+\infty$ on one component of $X\setminus\{(0,0)\}$
and $-\infty$ on the other component.
According to Theorems~10.1 and~12.1 of~\cite{scherkon},
  a pitchfork with $\Omega=\RR\times (0,b)$ exists if and only if $b\ge \pi$.
We conjecture that for each $b\ge \pi$, the pitchfork is unique up to rigid motions.

If $\Omega=\RR\times(0,b)$  and $S$ has exactly two points, then by Remark~\ref{periodic-remark},
one point (the origin)
 is on the line $y=0$
and the other point is on the line $y=b$.  
In this case, such a translator is called {\bf helicoid-like}.
Helicoid-like translators are described in~\cite{scherkon}, where it is proved that a helicoid-like translator
with $\Omega=\RR\times (0,b)$ exists if and only if $b<\pi$.
(We do not know whether, given $b$, the translator is unique up to rigid motion.)

Now suppose that $\Omega=\RR\times (0,b)$ and that $S$ contains $3$ or more points.
Then $S$ must contain more than one point on one edge of $\Omega$, say on the edge $y=0$.  
Then by Remark~\ref{periodic-remark},
\[
   S\cap \{x=0\} = \{(na,0): n\in \ZZ\}
\]
for some $a>0$, and $M$ is periodic with period $(2a,0,0)$.

If $S$ also contained a point $(k,b)$ on the side $y=b$, then by the periodicity, 
it would contain $(k+na,b)$
for every $n$.  That cannot happen according to Lemma~\ref{divergence-lemma} below.

Thus if $\Omega=\RR\times(0,b)$ is a strip and if $S$ contains more then $2$ points,
then $S=\{(na,0): n\in \ZZ\}$ for some $a>0$.  In this case, $M$ is the trident described 
in Theorem~\ref{main-theorem}.

In summary, we have

\begin{theorem}\label{semigraphical-theorem}
If $M$ is a semigraphical translator in $\RR^3$, then it is one of the following:
\begin{enumerate}[\upshape (1)]
\item a (doubly-periodic) Scherk translator,
\item a (singly-periodic) Scherkenoid,
\item a (singly-periodic) helicoid-like translator,
\item a pitchfork, 
\item a (singly-periodic) trident, or
\item\label{halfplane-item} (after a rigid motion) a translator containing $Z$ such that $M\setminus Z$
  is a graph over $\{(x,y): y\ne 0\}$.
\end{enumerate}
\end{theorem}

As mentioned above, we conjecture that Case~\eqref{halfplane-item} cannot
occur.

Of course pitchforks are not periodic, nor are the Case~\eqref{halfplane-item} examples (if they exist).

The proof of Theorem~\ref{semigraphical-theorem} was based on the following three
lemmas:

\begin{lemma}\label{wedge-lemma}
Let $\Omega=\{(r\cos\theta, r\sin\theta): r>0, \, 0<\theta < \beta\}$ where $0<\beta< \pi$.
There is no translator $\Sigma$ such that 
\begin{enumerate}
\item $\Sigma$ is a smooth, properly embedded manifold-with-boundary, the boundary
being $Z$, and
\item  $\Sigma\setminus Z$ is the graph a function $u:\Omega\to \RR$.
\end{enumerate}
\end{lemma}

\begin{proof}
Suppose to the contrary that such an $\Sigma$ exists.
Then $\Sigma$ and its vertical translates foliate $\Omega\times\RR$.  Let $\mathbf{n}$ be the downward
pointing unit vectorfield on $\Omega\times \RR$ that is perpendicular to the foliation.  Since the mean curvature
is a positive multiple of $\mathbf{n}$, $\Div\mathbf{n}$ is negative.
Let $W\subset\Omega\times\RR$ be a region with smooth boundary that is
transverse to $\Sigma$.
  Let $W^+$ and $\partial^+W$ be the portions of $W$ and of $\partial W$ that lie
above $\Sigma$.  Let $\nu$ be the outward pointing unit normal on $\partial (W^+)$. 
Then 
\begin{equation}\label{divergence-area-bound}
\begin{aligned}
0 
&\ge \int_{W^+} \Div \mathbf{n} \\
&=  \int_{W\cap\Sigma} \mathbf{n}\cdot\nu + \int_{\partial^+W}\mathbf{n}\cdot\nu \\
&\ge \area(W\cap \Sigma) - \area(\partial^+W)
\end{aligned}
\end{equation}
since $\mathbf{n}=\nu$ on $\Sigma\cap \Omega$ and since $|\mathbf{n}\cdot\nu|\le 1$.
Since $W$ is arbitrary, the area bound
 $\area(W\cap\Sigma)\le \area(\partial^+W)$ from~\eqref{divergence-area-bound} implies
that $\Sigma$ has finite entropy.
By Theorem~34 of \cite{white-boundaryMCF}, the tangent flow at infinity to the flow
\[
    t\in \RR \mapsto \Sigma - (0,0,t)
\]
 is a static, multiplicity-one halfplane.
Thus by Huisken monotonicity, $\Sigma$ is a flat halfplane, a contradiction.
\end{proof}

\begin{lemma}\label{divergence-lemma}
There is no $(2a,0)$-periodic translator
$
  u: \RR\times (0,b)\to \RR
$
such that 
\[
u(x,0)=u(k+x,b) = 
\begin{cases}
-\infty &\text{for $-a<x<0$, and} \\
\infty &\text{for $0<x<a$}.
\end{cases}
\]
\end{lemma}

\begin{proof}
Let $P$ be a fundamental parallelogram, e.g., the parallelogram
with corners $(0,0)$, $(2a,0)$, $(k,b)$, and $(2a+k,b)$.
Recall the translator equation
\begin{equation}\label{the-pde}
  \Div\xi = - (1 + |\D u|^2)^{-1/2}.
\end{equation}
where 
\[
   \xi = \frac{\D u}{\sqrt{1+|\D u|^2}}.
\]
By~\eqref{the-pde} and the divergence theorem,
\[
  \int_{\partial P} \xi\cdot\nn\,ds < 0,
\]
where $\nn$ is the outward pointing unit normal.
The integrals on the left and right sides of $P$ are equal and opposite and
so cancel each other out.  On the top and bottom edges of $P$,
the integrand is $1$ where $u=-\infty$ and $-1$ where $u=\infty$.
Thus the integral is $0$, a contradiction.

Note that because the vectorfield $\xi$ is bounded, the divergence theorem holds even though
there are isolated points (namely, the corners of the parallelogram) where $\xi$ is discontinuous.
\end{proof}

\begin{lemma}\label{one-line-lemma}
A translator $u:\{(x,y):y>0\} \to \RR$ cannot be periodic in the $x$-direction.
\end{lemma}

\begin{proof}
Otherwise, 
$
  (x,y)\in \RR\times (2\pi,3\pi) \mapsto \log(\sin y) - u(x,y)
$
would attain its maximum, violating the strong maximum principle.
\end{proof}

\end{document}